\pdfoutput=1 
\documentclass[10pt]{amsart}
\usepackage{amssymb}

\usepackage[theoremfont]{newpxtext}
\usepackage[scaled=1.05]{roboto} 
\usepackage[varbb, varg, smallerops]{newpxmath}
\usepackage[scaled=1.0]{inconsolata} 
\usepackage[cal=cm, calscaled=0.95, frak=euler, frakscaled=1.04]{mathalfa}
\usepackage[T1]{fontenc}

\usepackage{tikz}
\usepackage{tikz-cd}
\usetikzlibrary{arrows}
\tikzcdset{arrow style=tikz,
diagrams={>={Stealth[round,length=4pt,width=6pt,inset=3pt]}}}
\usepackage{graphicx}
\usepackage{hyperref}
\hypersetup{colorlinks=false}

\title[BC Rings]{Rings of Bounded Continuous Functions}
\date{19 February 2022}

\author{Yotam Svoray}
\address{Svoray: Department of  Mathematics,
Ben Gurion University, Be'er Sheva 84105, Israel}
\email{ysavorai@post.bgu.ac.il}

\author{Amnon Yekutieli}
\address{Yekutieli: Department of  Mathematics,
Ben Gurion University, Be'er Sheva 84105, Israel}
\email{amyekut@math.bgu.ac.il}

\newtheorem{thm}[equation]{Theorem}
\newtheorem{cor}[equation]{Corollary}
\newtheorem{prop}[equation]{Proposition}
\newtheorem{lem}[equation]{Lemma}
\theoremstyle{definition}
\newtheorem{dfn}[equation]{Definition}
\newtheorem{rem}[equation]{Remark}
\newtheorem{exa}[equation]{Example}

\newtheorem{que}[equation]{Question}

\newtheorem{conv}[equation]{Convention}

\numberwithin{equation}{section}

\newcommand{\iso}{\xrightarrow{%
\smash{\raisebox{-0.5ex}{\ensuremath{\scriptstyle \simeq  \mspace{2mu}}}}}}%

\newcommand{\xar}{\xrightarrow}

\newcommand{\sub}{\subseteq}
\newcommand{\opn}{\operatorname}
\newcommand{\cat}[1]{\operatorname{\mathsf{#1}}}

\newcommand{\cd}{\mspace{2.5mu}{\cdot}\mspace{2.5mu}} 

\newcommand{\ol}{\overline}
\newcommand{\rmitem}[1]{\item[\text{\textup{(#1)}}]}
\newcommand{\mfrak}[1]{\mathfrak{#1}}
\newcommand{\mcal}[1]{\mathcal{#1}}

\newcommand{\mbf}[1]{\mathbf{#1}}
\newcommand{\mrm}[1]{\mathrm{#1}}
\newcommand{\mbb}[1]{\mathbb{#1}}
\newcommand{\OO}{\mcal{O}}

\newcommand{\la}{\lambda}

\newcommand{\ga}{\gamma}
\newcommand{\ep}{\epsilon}
\newcommand{\ze}{\zeta}

\renewcommand{\a}{\mfrak{a}}

\newcommand{\p}{\mfrak{p}}
\newcommand{\q}{\mfrak{q}}
\newcommand{\m}{\mfrak{m}}
\newcommand{\n}{\mfrak{n}}

\newcommand{\bi}{\bsym{i}}

\newcommand{\K}{\mathbb{K}}
\newcommand{\R}{\mathbb{R}}
\newcommand{\Q}{\mathbb{Q}}
\newcommand{\Z}{\mathbb{Z}}
\newcommand{\N}{\mathbb{N}}
\newcommand{\C}{\mathbb{C}}

\newcommand{\bsym}[1]{\boldsymbol{#1}}

\newcommand{\ot}{\otimes}
\newcommand{\wtil}[1]{\widetilde{#1}}
\newcommand{\til}[1]{\tilde{#1}}

\newcommand{\what}[1]{\widehat{#1}}

\newcommand{\Ast}{\mathop{\scalebox{1.1}{\raisebox{-0ex}{$\ast$}}}}%

\newcommand{\lb}{\linebreak}
\newcommand{\abs}[1]{\lvert #1 \rvert}
\newcommand{\norm}[1]{\lVert #1 \rVert}
\renewcommand{\over}{\mspace{0.2mu} / \mspace{0mu}}
\newcommand{\xover}[1]{\mspace{0.0mu} /_{\mspace{-2mu} \mrm{#1}}%
\mspace{1.5mu}}
\newcommand{\lmsp}{\mspace{1.5mu}}
\newcommand{\ltsp}{\hspace{0.15em}}

\begin{document}

\begin{abstract}
We examine several classical concepts from topology and functional analysis, 
using methods of commutative algebra. We show that these various concepts are 
all controlled by {\em BC $\R$-rings} and their {\em maximal spectra}. 

A BC $\R$-ring is a ring $A$ that is isomorphic to the ring 
$\opn{F}_{\mrm{bc}}(X, \R)$ of {\em bounded continuous} $\R$-valued functions 
on some {\em compact} topological space $X$. These rings are not 
topologized; a morphism of BC $\R$-rings is just an $\R$-ring homomorphism. 

We prove that the category of BC $\R$-rings is dual to the category of compact 
topological spaces. Next we prove  that for {\em every} topological space $X$ 
the ring $\opn{F}_{\mrm{bc}}(X, \R)$ is a BC $\R$-ring. These theorems combined 
yield an algebraic construction of the 
{\em Stone-\v{C}ech Compactification} of an arbitrary topological space. 

There is a similar notion of {\em BC $\C$-ring}. Every BC $\C$-ring $A$ has a 
{\em canonical involution}. The {\em canonical hermitian subring} of $A$ is a 
BC $\R$-ring, and this is an equivalence of categories from BC $\C$-rings to 
BC $\R$-rings.  

Let $\K$ be either $\R$ or $\C$. We prove that a BC $\K$-ring $A$ has a {\em 
canonical norm} on it, making it into a Banach $\K$-ring. We then prove that 
the forgetful functor is an equivalence from Banach$^{\Ast}$ $\K$-rings (better 
known as commutative unital $\mrm{C}^{\Ast}$ $\K$-algebras) to BC $\K$-rings. 
The quasi-inverse of the forgetful functor endows a BC $\K$-ring with its 
canonical norm, and the canonical involution when $\K = \C$. 

{\em Stone topological spaces}, also known as {\em profinite topological 
spaces}, are traditionally related to {\em boolean rings} -- this is {\em Stone 
Duality}. We give a BC ring characterization of Stone spaces. 
From that we obtain a very easy proof of the fact that the 
Stone-\v{C}ech Compactification of a discrete space is a Stone space. 

Most of the results in this paper are not new. However, most of our proofs seem 
to be new -- and our methods could potentially lead to genuine 
progress related to these classical topics.  
\end{abstract}

\maketitle

\tableofcontents


\setcounter{section}{-1}
\section{Introduction}

In this paper we look at the classical -- namely middle 20th century -- 
concepts of {\em compact topological spaces}, {\em Stone-\v{C}ech 
Compactification} and {\em commutative $\mrm{C}^{\Ast}$-algebras} from an 
algebraic geometer's point of view, using the language of categories and 
commutative algebra. Our goal in this paper is clarify some of the 
ideas involved, and to find an alternative uniform treatment. This is because, 
in our opinion, the traditional presentation in textbooks on topology and 
functional analysis tends to obscure some of the interesting aspects of 
these classical concepts. 

Our treatment is based on the unifying concept of {\em BC ring}, 
a concept that we introduce in this paper (see Definition \ref{dfn:470} below). 
The algebraic properties of BC rings turn out to control compact topological 
spaces, Stone-\v{C}ech compactification, and commutative $\mrm{C}^{\Ast}$ 
algebras over $\R$ and $\C$. See diagram (\ref{eqn:515}) for the web of 
category equivalences. Of the five categories appearing in this diagram, the 
category $\cat{Rng} \xover{bc} \R$ of BC $\R$-rings, and $\R$-ring 
homomorphisms between them (these are abstract rings, no continuity is 
involved), seems to be the easiest to work with. 

Let us say a few words about originality. Most of the results in 
this paper are not new at all. Some of them can even be found in 
textbooks, and a few are merely rephrasings of classical theorems.
Yet most of our proofs are probably new, and could potentially lead to genuine 
progress. 

In order to make our paper accessible to mathematicians working in topology or 
functional analysis, who might lack sufficient background knowledge of 
categories and commutative algebra, we have included a recollection of the 
necessary material in Section \ref{sec:prlim}, with textbook references. Very 
brief explanations of the terminology are inserted here, in the Introduction, 
to facilitate reading the main theorems. 

We call a topological space $X$ {\em compact} if it is Hausdorff 
and quasi-compact (namely it has the finite open subcovering property). The 
{\em category of topological spaces}, with continuous maps, is denoted by 
$\cat{Top}$, and its full subcategory on the {\em compact spaces} is 
$\cat{Top}_{\mrm{cp}}$. 

All rings in this paper are commutative (and unital). The category of rings is 
denoted by $\cat{Rng}$. Throughout the Introduction we let $\K$ denote either 
of the fields $\R$ or $\C$. By a {\em $\K$-ring} we mean a ring $A$ equipped 
with a ring homomorphism $\K \to A$. (Traditionally $A$ would be called a 
commutative associative unital $\K$-algebra.) The category of $\K$-rings, with 
$\K$-ring homomorphisms, is $\cat{Rng} \over \K$.
It is important to emphasize that the objects of the category 
$\cat{Rng} \over \K$ do not carry topologies, and hence there is no continuity 
condition on the morphisms in $\cat{Rng} \over \K$.

For a topological space $X$ we write $\opn{F}_{\mrm{bc}}(X, \K)$ 
for the ring of {\em bounded continuous functions} $a : X \to \K$, where the 
field $\K$ (which is either $\R$ or $\C$) 
is endowed with its standard norm. As $X$ changes we obtain a 
contravariant functor
\begin{equation} \label{eqn:315}
\opn{F}_{\mrm{bc}}(-, \K) : \cat{Top} \to \cat{Rng} \over \K . 
\end{equation}

The next definition presents the class of rings that is the focus of our 
paper. 

\begin{dfn} \label{dfn:470}
Let $\K$ be either of the fields $\R$ or $\C$.
A $\K$-ring $A$ is called a {\em BC $\K$-ring} if there is an isomorphism of 
$\K$-rings $A \cong \opn{F}_{\mrm{bc}}(X, \K)$ 
for some {\em compact topological space} $X$. 
The full subcategory of $\cat{Rng} \over \K$ on the BC rings is denoted by
$\cat{Rng} \xover{bc} \K$. 
\end{dfn}

Given a ring $A$, its prime spectrum $\opn{Spec}(A)$ is considered here only 
as a topological space, with the Zariski topology; we ignore the structure 
sheaf of the affine scheme $\opn{Spec}(A)$. The set of maximal ideals of $A$ is 
$\opn{MSpec}(A)$, and it is given the induced subspace topology from 
$\opn{Spec}(A)$. 

The prime spectrum is a contravariant functor 
\begin{equation} \label{eqn:316}
\opn{Spec} : \cat{Rng} \over \K \to \cat{Top} , 
\end{equation}
sending a $\K$-ring homomorphism $\phi : A \to B$ to the map 
\begin{equation} \label{eqn:317}
\opn{Spec}(\phi) : \opn{Spec}(B) \to \opn{Spec}(A) , \quad
\q \mapsto \phi^{-1}(\q) 
\end{equation}
in $\cat{Top}$. However, $\opn{MSpec}$ is not a contravariant functor
$\cat{Rng} \over \K \to \cat{Top}$; see Example \ref{exa:365}.

In Sections \ref{BC-rings} and \ref{BC-C-rings} we show that for a BC $\K$-ring 
$A$ and a maximal ideal $\m \sub A$, the residue field is $A / \m = \K$. 
This implies that given a homomorphism $\phi : A \to B$ of BC $\K$-rings, and 
a maximal ideal $\n \sub B$, the ideal $\phi^{-1}(\n) \sub A$ is maximal.
Therefore we get a map 
\begin{equation} \label{eqn:318}
\opn{MSpec}(\phi) : \opn{MSpec}(B) \to \opn{MSpec}(A) , \quad
\n \mapsto \phi^{-1}(\n) . 
\end{equation}    
We see that there is a contravariant functor 
\begin{equation} \label{eqn:470}
\opn{MSpec} : \cat{Rng}\xover{bc} \K \to \cat{Top} .
\end{equation}

Here is the first main result of the paper. 

\begin{thm}[Duality] \label{thm:305}
Let $\K$ be either $\R$ or $\C$. 
The contravariant functor 
\[ \opn{F}_{\mrm{bc}}(-, \K) : \cat{Top}_{\mrm{cp}} \to 
\cat{Rng}\xover{bc} \K \]
is a duality (namely a contravariant equivalence of categories), with 
quasi-inverse $\opn{MSpec}$.
\end{thm}

This result is repeated as Theorems \ref{thm:203} and \ref{thm:420} 
(for the real and complex cases, respectively), and 
proved there. A non-categorical phrasing of this result can be found in 
\cite[pages 15-16]{Wa}, where it is attributed to Stone. 
The fact that the functor $\opn{F}_{\mrm{bc}}(-, \R)$ is fully faithful is 
sometimes called the {\em Gelfand-Kolgomorov Theorem}.  

We see that BC $\K$-rings play a role similar to the role that boolean rings 
have in the context of Stone spaces; the corresponding duality there is 
called {\em Stone duality}, see \cite[Corollary II.4.4]{Jo}. Our 
Theorem \ref{thm:517} below makes the precise connection between BC rings and
Stone spaces. 

Our second main result says that BC rings occur in much greater generality than 
Definition \ref{dfn:470} indicates. 

\begin{thm} \label{thm:306} 
Let $\K$ be either $\R$ or $\C$. 
For an arbitrary topological space $X$, the $\K$-ring
$\opn{F}_{\mrm{bc}}(X, \K)$ is a BC ring. 
\end{thm}

This theorem is repeated as Theorems \ref{thm:415} and \ref{thm:456} in the 
body of the paper, for the case $\K = \R$ and $\K = \C$ respectively.

A {\em Stone-\v{C}ech Compactification} (SCC) of a topological space $X$ is a 
pair $(\bar{X}, \opn{c}_X)$, consisting of a compact topological space 
$\bar{X}$ 
and a continuous map $\opn{c}_X : X \to \bar{X}$, which is universal for 
continuous maps from $X$ to compact spaces; see Definition \ref{dfn:131} for 
more details. It is clear that an SCC of $X$, if it exists, is unique up to a 
unique isomorphism. The categorical property of SCC is stated in Proposition 
\ref{prop:285}. There are several existence proofs of SCC in the literature, 
and we mention some of them in Remark \ref{rem:285}. 

Consider a topological space $X$.
For a point $x \in X$ and a function $a \in \opn{F}_{\mrm{bc}}(X, \K)$
the evaluation of $a$ at $x$ is $\opn{ev}_x(a) := a(x) \in \K$. The {\em 
algebraic reflection map} is the map 
\begin{equation} \label{eqn:320}
\opn{refl}^{\lmsp \mrm{alg}}_X : X \to \opn{MSpec} 
\bigl( \opn{F}_{\mrm{bc}}(X, \R) \bigr) ,
\quad \opn{refl}^{\lmsp \mrm{alg}}_X(x) := \opn{Ker}(\opn{ev}_x) 
\end{equation}
in $\cat{Top}$. Our third main theorem is next.

\begin{thm}[Algebraic SCC] \label{thm:309}
Given a topological space $X$, consider the topological space 
$\bar{X}^{\mrm{alg}} := \opn{MSpec} \bigl( \opn{F}_{\mrm{bc}}(X, \R) \bigr)$
and the algebraic reflection map 
$\opn{refl}^{\mrm{\lmsp alg}}_X : X \to \bar{X}^{\mrm{alg}}$. Then the pair 
$\bigl( \bar{X}^{\mrm{alg}}, \opn{refl}^{\mrm{\lmsp alg}}_X \bigr)$ 
is a Stone-\v{C}ech Compactification of $X$. 
\end{thm}

This is Theorem \ref{thm:285} in the body of the paper. The proof is an easy 
consequence of Theorems \ref{thm:305} and \ref{thm:306}. 

The next theorem is about Banach rings. 

\begin{thm}[Canonical Norm] \label{thm:515}
Let $\K$ be either $\R$ or $\C$. 
Every BC $\K$-ring $A$ admits a unique norm $\norm{-}$, called the 
{\em canonical norm}, satisfying the three conditions below. 
\begin{itemize}
\rmitem{i} Given a homomorphism $\phi : A \to B$ in 
$\cat{Rng} \xover{bc} \K$, and an element $a \in A$, the inequality
$\norm{\phi(a)} \leq \norm{a}$ holds. 

\rmitem{ii} For the $\K$-ring $\opn{F}_{\mrm{bc}}(X, \K)$
of bounded continuous functions on a topological space $X$, 
the canonical norm is the sup norm. 

\rmitem{iii} The canonical norm makes $A$ into a Banach $\K$-ring.  
\end{itemize}
\end{thm}

This theorem is repeated as Theorem \ref{thm:486} for the real case, and as 
Theorem \ref{thm:505} for the complex case. 

An {\em involution} of a $\C$-ring $A$ is a ring automorphism 
$(-)^{\Ast}$ of $A$ such that $(a^{\Ast})^{\Ast} = a$ and 
$(\la \cd a)^{\Ast} = \bar{\la} \cd a^{\Ast}$ 
for all  $a \in A$ and $\la \in \C$. Here $\bar{\la}$ is the complex 
conjugate of $\la$. The prototypical example is the involution 
$(-)^{\Ast}$ of the ring $A := \opn{F}_{\mrm{bc}}(X, \C)$, where $X$ is some 
topological space, whose formula is 
\begin{equation} \label{eqn:475}
a^{\Ast}(x) = \ol{a(x)}
\end{equation}
for all $a \in A$ and $x \in X$. 

By a {\em Banach$^{\Ast}$ $\C$-ring} we mean a $\C$-ring $A$, equipped with an  
involution $(-)^{\Ast}$ and a norm $\norm{-}$, such that $A$ is a Banach ring 
with this norm, and also $\norm{a \cd a^{\Ast}} = \norm{a}^2$ for all 
$a \in A$. This object is commonly known as a {\em commutative unital 
$\mrm{C}^{\Ast}$ algebra over $\C$}.

\begin{thm}[Canonical Involution] \label{thm:470}
Every BC $\C$-ring $A$ admits a unique involution $(-)^{\Ast}$, called the 
{\em canonical involution}, satisfying the three conditions below. 
\begin{itemize}
\rmitem{i} Given a homomorphism $\phi : A \to B$ in 
$\cat{Rng}\xover{bc} \C$, and an element $a \in A$, there is equality
$\phi(a^{\Ast}) = \phi(a)^{\Ast}$ in $B$.

\rmitem{ii} For the $\C$-ring $\opn{F}_{\mrm{bc}}(X, \C)$
of bounded continuous functions on a topological space $X$, 
the canonical involution is the one from equation (\ref{eqn:475}). 

\rmitem{iii} The $\C$-ring $A$, equipped with the canonical involution 
$(-)^{\Ast}$ and the canonical norm $\norm{-}$ from Theorem \ref{thm:515}, is a 
Banach$^{\Ast}$ $\C$-ring. 
\end{itemize}
\end{thm}

This theorem is repeated as Theorem \ref{thm:435} and part 
of Theorem \ref{thm:505} . 

For a BC $\C$-ring $A$, with its 
canonical involution $(-)^{\Ast}$, the {\em canonical hermitian subring} is 
the $\R$-ring 
$A_0 := \{ a \in A \mid a^{\Ast} = a \}$. 
The next theorem is a combination of Corollary \ref{cor:445} and Theorem 
\ref{thm:475} in the body of the paper.

\begin{thm} \label{thm:476}  
There is an equivalence of categories
\[ H : \cat{Rng}\xover{bc} \C \to \cat{Rng}\xover{bc} \R , \]
which sends a BC $\C$-ring $A$ to its canonical hermitian subring $H(A) = A_0$.
The quasi-inverse of $H$ is the induction functor 
\[ I : \cat{Rng}\xover{bc} \R \to \cat{Rng}\xover{bc} \C \, , \ \ 
I(A_0) := \C \ot_{\R} A_0 . \]
\end{thm}

A {\em Banach$^{\Ast}$ $\R$-ring} is an $\R$-ring $A$, equipped with a norm 
$\norm{-}$, such that $A$ is a Banach $\R$-ring with this norm, and for every 
$a \in A$ there is equality $\norm{a^2} = \norm{a}^2$, 
and the element $1 + a^2$ is invertible in $A$. 

\begin{thm} \label{thm:472}
Let $\K$ be either $\R$ or $\C$. 
There is an equivalence of categories
\[ F : \cat{BaRng}^{\Ast} \over \K \to \cat{Rng}\xover{bc} \K , \]
which forgets the norm, and forgets the involution when $\K = \C$.
The quasi-inverse 
\[ G : \cat{Rng}\xover{bc} \K \to \cat{BaRng}^{\Ast} \over \K \]
endows a BC $\K$-ring with its canonical norm, and with its canonical 
involution when $\K = \C$. 
\end{thm}

The equivalences of categories from Theorems \ref{thm:305}, \ref{thm:476} and 
\ref{thm:472} are exhibited in the following diagram, which is commutative up 
to isomorphisms of functors. 

\begin{equation} \label{eqn:515}
\begin{tikzcd} [column sep = 6ex, row sep = 9ex]
\cat{BaRng}^{\Ast} \over \C
\ar[d, shift right =1.5ex, bend right = 20, "{F}"']
&
&
\cat{BaRng}^{\Ast} \over \R
\ar[d, shift left = 1.5ex, bend left = 20, "{F}"]
\\
\cat{Rng} \xover{bc} \C
\ar[u, shift right= 1.5ex, bend right = 20, "{G}"']
\ar[rr, shift left = 0.3ex, bend left = 10, "{H}"]
\ar[dr, shift left = 0ex,  bend left = 10, start anchor = south, 
end anchor = north west, "{\opn{MSpec}}" near end]
&
&
\cat{Rng} \xover{bc} \R
\ar[u, shift left = 1.5ex, bend left = 20, "{G}"]
\ar[ll, shift left = 0.3ex, bend left = 10, "{I}"]
\ar[dl, shift right = 0ex,  bend right = 05, start anchor = south, 
end anchor = north east, "{\opn{MSpec}}"' near end]
\\
&
(\cat{Top}_{\mrm{cp}})^{\mrm{op}}
\ar[ul, shift left = 1.5ex, bend left = 10, 
start anchor = north west, end anchor = south, "{\opn{F}_{\mrm{bc}}(-, \C)}"]
\ar[ur, shift right = 1.5ex, bend right = 10, start anchor = north east, 
end anchor = south, "{\opn{F}_{\mrm{bc}}(-, \R)}"']
\end{tikzcd}
\end{equation}

The last results to be mentioned in the Introduction concern {\em Stone 
topological spaces}, also known as {\em profinite topological spaces}. 

Let $X$ be a topological space, with function ring
$A := \opn{F}_{\mrm{bc}}(X, \R)$. A function $a \in A$ is 
called a {\em step function} if is takes only finite many values. The subring 
of $A$ consisting of step functions is denoted by $A_{\mrm{stp}}$. 

\begin{thm} \label{thm:517} 
Let $A$ be a BC $\R$-ring and let $X := \opn{MSpec}(A)$. 
The following two conditions are equivalent:
\begin{itemize}
\rmitem{i} $X$ is a Stone space. 

\rmitem{ii} The subring $A_{\mrm{stp}}$ is dense in $A$, with respect to the 
canonical norm of $A$. 
\end{itemize}
\end{thm}

This is repeated as Theorem \ref{thm:490} in Section \ref{stone} and proved 
there. 

\begin{cor} \label{cor:517}
Suppose $X$ is a discrete topological space, with Stone-\v{C}ech 
Compactification $\bar{X}$. Then $\bar{X}$ is a Stone topological space. 
\end{cor}

This fact is well-known of course, but all previous proofs of it that  
we saw are involved and indirect. Our proof (this is repeated as Corollary 
\ref{cor:515}) is an easy consequence of Theorem \ref{thm:517} -- a discrete 
space $X$ admits plenty of step functions.  

We end the Introduction with two questions. 

\begin{que} \label{que:400}
Is there a theory of "BC rings" for {\em commutative nonarchimedean Banach 
rings}? Is there result analogous to Theorem \ref{thm:472}, but with 
$\what{\Q}_p$ instead of $\R$ or $\C$~?
\end{que}

\begin{que} \label{que:465}
Is there a useful theory of "noncommutative BC rings", and a 
noncommutative version of Theorem \ref{thm:472} for {\em 
noncommutative Banach rings}?
\end{que}

\vspace{1em}
\medskip \noindent 
{\em Acknowledgments}. We wish to thank Nicolas Addington, Moshe Kamenski,
Assaf Hasson, Eli Shamovich, Francesco Saettone, Shirly Geffen, Ilan Hirshberg 
and Ken Goodearl for useful discussions.

\section{Preliminaries} 
\label{sec:prlim}

In this section we review some of the concepts on categories and commutative 
rings that will be used in our paper, for the benefit of readers who are not 
familiar with these subjects. This section also serves to introduce 
notation. Textbook references are provided. 

A {\em category} $\cat{C}$ consists of a set of objects $\opn{Ob}(\cat{C})$, 
and for each pair of objects $C_0, C_1 \in \opn{Ob}(\cat{C})$
a set of morphisms 
$\opn{Hom}_{\cat{C}}(C_0, C_1)$. A morphism 
$f \in \opn{Hom}_{\cat{C}}(C_0, C_1)$
is depicted by $f : C_0 \to C_1$. Every object 
$C \in \opn{Ob}(\cat{C})$
has an identity automorphism $\opn{id}_C : C \to C$.
There is an operation of composition 
\[ \opn{Hom}_{\cat{C}}(C_1, C_2) \times 
\opn{Hom}_{\cat{C}}(C_0, C_1) \to 
\opn{Hom}_{\cat{C}}(C_0, C_2) , \quad (f_2, f_1) \mapsto f_2 \circ f_1 \]
for each triple of objects $C_0, C_1, C_2$. 
The composition is associative, and the identity automorphisms satisfy
$\opn{id}_{C_1} \circ \, f = f \circ \opn{id}_{C_0}$ 
for all morphisms $f : C_0 \to C_1$. A morphism $f : C_0 \to C_1$ is called an 
{\em isomorphism} if there is a morphism $g : C_1 \to C_0$ 
such that $g \circ f = \opn{id}_{C_0}$ and 
$f \circ g = \opn{id}_{C_1}$. In this case $g$ is unique, it is called the 
inverse of $f$, and one writes $f^{-1} := g$. 
As customary, we will use the shorthand $C \in \cat{C}$ for an object $C$. 
We shall ignore set theoretical issues in this paper, and implicitly rely on 
the concept of {\em Grothendieck universe}, as explained in 
\cite[Section I.6]{ML} and in \cite[Section 1.1]{Ye}. 

Given categories $\cat{C}$ and $\cat{D}$, a {\em functor}
$F : \cat{C} \to \cat{D}$
is the data of a function
$F_{\mrm{ob}} : \opn{Ob}(\cat{C}) \to \opn{Ob}(\cat{D})$,
and for every pair of objects 
$C_0, C_1 \in \opn{Ob}(\cat{C})$ a function
\begin{equation} \label{eqn:360}
F_{C_0, C_1} : \opn{Hom}_{\cat{C}}(C_0, C_1) \to 
\opn{Hom}_{\cat{D}} \bigl( F_{\mrm{ob}}(C_0), F_{\mrm{ob}}(C_1) \bigr) . 
\end{equation}
The functions $F_{C_0, C_1}$ must respect compositions and identities. 
When there is no ambiguity, and in order to simplify the notation, we shall use 
the symbol $F$ both for $F_{\mrm{ob}}$ and for $F_{C_0, C_1}$. 
Each category $\cat{C}$ has its identity functor $\opn{Id}_{\cat{C}}$, which 
acts identically on objects and morphisms of $C$. 
Functors 
$F : \cat{C} \to \cat{D}$ and $G : \cat{D} \to \cat{E}$
can be composed in the obvious way, and 
$\opn{Id}_{\cat{D}} \circ \, F = F \circ \opn{Id}_{\cat{C}}$. 
Sometimes a functor is called a {\em covariant functor}. 

Suppose 
$F, G : \cat{C} \to \cat{D}$
are functors. A {\em morphism of functors} (also called a {\em natural 
transformation}) $\eta : F \to G$ is a collection
$\{ \eta_{C} \}_{C \in \opn{Ob}(\cat{C})}$ of morphisms
$\eta_C : F(C) \to G(C)$ in $\cat{D}$, such that for every morphism 
$f : C_0 \to C_1$ in $\cat{D}$ the diagram 
\[ \begin{tikzcd} [column sep = 10ex, row sep = 6ex]
F(C_0)
\ar[r, "{F(f)}"]
\ar[d, "{\eta_{C_0}}"']
&
F(C_1)
\ar[d, "{\eta_{C_1}}"]
\\
G(C_0)
\ar[r, "{G(f)}"]
&
G(C_1)
\end{tikzcd} \]
in $\cat{D}$ is commutative. 
The morphism $\eta$ is called an {\em isomorphism of functors} if for every 
object $C \in \cat{C}$ the morphism 
$\eta_C : F(C) \to G(C)$ is an isomorphism. 

A functor $F : \cat{C} \to \cat{D}$ is called {\em full} (resp.\ {\em 
faithful}) if for every pair of objects $C_0, C_1 \in \cat{C}$ the function 
$F_{C_0, C_1}$ in (\ref{eqn:360}) is surjective (resp.\ injective). 
The functor $F : \cat{C} \to \cat{D}$ is called an {\em equivalence} if 
there exists a functor $G : \cat{D} \to \cat{C}$,
and isomorphisms of functors 
$\eta : \opn{Id}_{\cat{C}} \iso G \circ F$ and 
$\ze : \opn{Id}_{\cat{D}} \iso F \circ G$.
In this case $G$ is called a quasi-inverse of $F$, and it is unique up to a 
unique isomorphism of functors. 
It is not hard to show that $F : \cat{C} \to \cat{D}$ is an equivalence iff it 
is both fully faithful (i.e.\ full and faithful) and essentially surjective on 
objects (i.e.\ for every $D \in \cat{D}$ there exists some $C \in \cat{C}$ 
with an isomorphism $F(C) \iso D$ in $\cat{D}$).  

A {\em contravariant functor} $F : \cat{C} \to \cat{D}$ 
is made up of functions $F_{\mrm{ob}}$ and  
$F_{C_0, C_1}$, but the target sets in (\ref{eqn:360}) are reversed; namely 
for each morphism $f : C_0 \to C_1$ in $\cat{C}$ there is a morphism 
$F(f) : F(C_1) \to F(C_0)$ in $\cat{D}$. The contravariant functor $F$ must 
respect compositions, in the reversed order, and identities. 
If $G : \cat{D} \to \cat{E}$ is another contravariant functor, then the 
composition $G \circ F : \cat{C} \to \cat{E}$ is a (covariant) functor. 

A category $\cat{C}$ gives rise to the {\em opposite category} 
$\cat{C}^{\mrm{op}}$, which has the same objects as $\cat{C}$, but the morphism 
sets and the compositions are reversed. There is a contravariant functor 
$\opn{Op}_{\cat{C}} : \cat{C} \to \cat{C}^{\mrm{op}}$, 
which is the identity on object sets, and the identity (in disguise)
\[ \opn{Op}_{\cat{C}, C_0, C_1} : \opn{Hom}_{\cat{C}}(C_0, C_1) \iso 
\opn{Hom}_{\cat{C}^{\mrm{op}}}(C_1, C_0) \]
on morphism sets. There is an equality of functors 
$\opn{Op}_{\cat{C}^{\mrm{op}}} \circ \opn{Op}_{\cat{C}} = \opn{Id}_{\cat{C}}$.
Every contravariant functor 
$F : \cat{C} \to \cat{D}$
can be expressed uniquely as 
$F = F^{\mrm{op}} \circ \opn{Op}_{\cat{C}}$, where 
$F^{\mrm{op}} : \cat{C}^{\mrm{op}} \to \cat{D}$ 
is the functor $F^{\mrm{op}} := F \circ  \opn{Op}_{\cat{C}^{\mrm{op}}}$;
see the commutative diagram (\ref{eqn:460}). 
The formula $F \mapsto F^{\mrm{op}}$ gives a bijection between the set of 
contravariant functors $\cat{C} \to \cat{D}$ and the set of functors 
$\cat{C}^{\mrm{op}} \to \cat{D}$, and we are going to use this bijection often.
\begin{equation} \label{eqn:460}
\begin{tikzcd} [column sep = 10ex, row sep = 6ex]
\cat{C}^{\mrm{op}}
\ar[r, "{ \opn{Op}_{\cat{C}^{\mrm{op}}} }"]
\ar[dr, "{F^{\mrm{op}}}"']
&
\cat{C}
\ar[d, "{F}"]
\\
&
\cat{E}
\end{tikzcd} 
\end{equation}

Let $\cat{C}$ be a category. A {\em full subcategory} $\cat{B}$ of $\cat{C}$ 
is a category $\cat{B}$ such that 
$\opn{Ob}(\cat{B}) \sub \opn{Ob}(\cat{C})$,
and for every pair of objects $C_0, C_1 \in \opn{Ob}(\cat{B})$ there is 
equality 
\[ \opn{Hom}_{\cat{B}}(C_0, C_1) = \opn{Hom}_{\cat{C}}(C_0, C_1) . \]
The identity morphisms and the composition of $\cat{B}$ are those of $\cat{C}$.
Thus the inclusion functor $\opn{Inc} : \cat{B} \to \cat{C}$
is fully faithful. 

For a detailed study of categories and functors in general we 
recommend the books \cite{HS} and \cite{ML}.

Now let us turn attention to a few categories that will play important roles in 
our paper. First there is the category $\cat{Set}$ of sets, whose 
objects are the sets (with an implicit bound on size, in terms of Grothendieck 
universes). The morphisms $f : S \to T$ in $\cat{Set}$ are the functions. 

The category of topological spaces and continuous maps between them is 
$\cat{Top}$. There is a forgetful functor 
$\cat{Top} \to \cat{Set}$, which sends a topological space to its 
underlying set, and does nothing to the morphisms. 

Recall that a topological space $X$ is called {\em Hausdorff} if for every pair 
of distinct points $x, y \in X$ there are open subsets 
$U, V \sub X$ such that $x \in U$, $y \in V$ and $U \cap V = \varnothing$. 
All metric spaces are Hausdorff, but there are many important Hausdorff 
topological spaces that are not metrizable (i.e.\ they do not admit metrics 
that induce their given topologies).  

A topological space $X$ is called {\em quasi-compact} 
if it has the finite open subcovering property. By this we mean that given an 
open covering 
$X = \bigcup_{i \in I} U_i$, indexed by some set $I$, there exists a finite 
subset $I_0 \sub I$ such that $X = \bigcup_{i \in I_0} U_i$.
The adjective quasi-compact is prominent in algebraic geometry, but we have not 
seen it used in classical topology publications.  

In our paper will adhere to this convention: 

\begin{conv} \label{conv:452}
A topological space $X$ is called {\em compact} if it is both Hausdorff 
and quasi-compact. 
\end{conv}

The full subcategory of $\cat{Top}$ on the compact topological spaces is denoted
by $\cat{Top}_{\mrm{cp}}$.
We will also adhere to the next convention regarding rings:

\begin{conv} \label{conv:451}
A {\em ring} means, by default, a unital commutative ring. 
A ring homomorphism $\phi : A \to B$ must respect units.
\end{conv}

The category of (commutative) rings and ring homomorphisms is denoted by 
$\cat{Rng}$. 
All rings in Sections \ref{sec:prlim}-\ref{sec:SCC} of this paper are 
without a topology; exceptions are the fields $\R$ and $\C$, which are 
sometimes seen as topological rings with their standard norm topologies, and 
when this happens we shall state it explicitly. In Sections
\ref{real-B}-\ref{sec:inv-c-b} we shall also deal with Banach rings, 
and again the presence of a norm will be stated explicitly.  

\begin{dfn} \label{dfn:365}
Fix a ring $\K$.
\begin{enumerate}
\item A $\K$-ring is a ring $A$ equipped with a ring
homomorphism $\opn{str}_{A} : \K \to A$, called the structural homomorphism.

\item Suppose $A$ and $B$ are $\K$-rings. A $\K$-ring homomorphism
$\phi : B \to C$ is a ring homomorphism $\phi$ satisfying
$\phi \circ \opn{str}_{A} = \opn{str}_{B}$.

\item The category of $\K$-rings and $\K$-ring homomorphisms is denoted by
$\cat{Rng} \over \K$.
\end{enumerate}
\end{dfn}

The structural homomorphism $\opn{str}_{A}$ shall usually remain implicit, and
we shall just talk about the $\K$-ring $A$. Most textbooks use the expression
``unital associative commutative $\K$-algebra'' to mean an $\K$-ring.
Since every ring $A$ admits a unique ring homomorphism $\Z \to A$, it follows 
that there is equality $\cat{Rng} \over \Z = \cat{Rng}$. 
In our paper we are going to be mostly concerned with $\R$-rings.

The last topic to review is {\em affine schemes}. Actually, we won't require 
the full structure of the affine scheme $(X, \OO_X) = \opn{Spec}(A)$
as a locally ringed space. The structure sheaf $\OO_X$ is going to be
ignored, and we shall only care about the underlying topological space $X$, 
whose definition we now recall. In this definition we introduce some new 
notation.

\begin{dfn}[Zariski Topology] \label{dfn:460}
Let $A$ be a ring. 
\begin{enumerate}
\item The {\em prime spectrum} of $A$ is the set $\opn{Spec}(A)$ of prime 
ideals of $A$. Let us write $X := \opn{Spec}(A)$. 

\item Given an ideal $\a \sub A$, the subset
\[ \opn{Zer}_X(\a) := 
\{ \lmsp \p \in X \mid \a \sub \p \lmsp \} \sub X \]
is called the {\em closed subset of $X$ defined by the ideal $\a$}. 

\item The {\em Zariski topology} of $X$ is the topology in which the closed  
subsets are $\opn{Zer}_X(\a)$, as $\a$ runs over all the ideals $\a \sub A$. 

\item If $\a = (a)$ is a principal ideal, then we write 
$\opn{Zer}_X(a) :=\opn{Zer}_X(\a)$. 

\item For an element $a \in A$ we write 
\[ \opn{NZer}_X(a) = \{ \lmsp \p \in X \mid a \notin \p \lmsp \} \sub X . \]
This is is called the {\em principal open subset defined by the element $a$}.
\end{enumerate}
\end{dfn}

Obviously the principal open subset $\opn{NZer}_X(a)$ is the complement of the 
principal closed subset $\opn{Zer}_X(a)$. It is easy to verify that 
\begin{equation} \label{eqn:365}
\opn{Zer}_X(\a) = \bigcap_{a \in \a} \, \opn{Zer}_X(a) .
\end{equation}
This formula implies that the principal 
open subsets of $X$ form a basis of the Zariski topology; i.e.\ for every
open subset $U \sub X$ and for every point $x \in U$ there is some element
$a \in A$ such that
$x \in \opn{NZer}_X(a) \sub U$.
It is known that the topological space
$X = \opn{Spec}(A)$ is quasi-compact. 

Suppose $\phi : A \to B$ is a ring homomorphism. If $\q \sub B$ is a prime 
ideal, then $\p := \phi^{-1}(\q) \sub A$ is also a prime ideal. The resulting 
function
\begin{equation} \label{eqn:370}
\opn{Spec}(\phi) : \opn{Spec}(B) \to \opn{Spec}(A) , \quad
\opn{Spec}(\phi)(\q) := \phi^{-1}(\q)
\end{equation}
is actually continuous (this is an easy exercise).
In this way we obtain a functor
\begin{equation} \label{eqn:368}
\opn{Spec} : \cat{Rng}^{\mrm{op}} \to \cat{Top} .
\end{equation}

For more information on the Zariski topology of affine schemes see the books 
\cite{AK} or \cite{Ei}.

\begin{dfn} \label{dfn:370}
Let $A$ be a ring. The set of maximal ideals of $A$ is denoted by
$\opn{MSpec}(A)$. The set $\opn{MSpec}(A)$ is a subset of $\opn{Spec}(A)$,
and we give it the induced Zariski topology.
\end{dfn}

\begin{prop} \label{prop:500}
The following are equivalent for a ring $A$. 
\begin{itemize}
\rmitem{i} $A = 0$. 

\rmitem{ii} $\opn{MSpec}(A) = \varnothing$.
\end{itemize}
\end{prop}

\begin{proof}
The implication (i) $\Rightarrow$ (ii) is trivial. 
The reverse implication is a standard exercise using Zorn's Lemma. 
\end{proof}

\begin{prop} \label{prop:501}
Let $A$ be a ring and let $X := \opn{MSpec}(A)$. 
The following are equivalent for an element $a \in A$. 
\begin{itemize}
\rmitem{i} $a$ is invertible in $A$. 

\rmitem{ii} $\opn{Zer}_X(a) = \varnothing$.
\end{itemize}
\end{prop}

\begin{proof}
The implication (i) $\Rightarrow$ (ii) is trivial. 
As for the reverse implication let $\a \sub A$ be the ideal generated by $a$,
and let $\bar{A} := A / \a$. It is easy to see that there is a canonical 
bijection $\opn{Zer}_X(a) \cong \opn{MSpec}(\bar{A})$. Using Proposition 
\ref{prop:500} we see that condition (ii) implies that $\bar{A} = 0$; and this 
means that $\a = A$ and that $a$ is invertible.
\end{proof}

\begin{prop} \label{prop:395}
For every ring $A$ the topological space $\opn{MSpec}(A)$ is quasi-compact. 
\end{prop}

\begin{proof}
Write $X := \opn{Spec}(A)$ and $X_{\mrm{max}} := \opn{MSpec}(A)$. 
Suppose $X_{\mrm{max}} = \bigcup_{i \in I} U_i$ is an open covering.
We need to prove that this has a finite subcovering. 
 
Since $X_{\mrm{max}}$ has the subspace topology induced from $X$, 
for every index $i$ there is some open subset $V_i \sub X$ such that 
$U_i = V_i \cap X_{\mrm{max}}$. Let $Z$ be the complement in $X$ of the open 
set $\bigcup_{i \in I} V_i$. Because $X$ has the Zariski topology, there is 
some ideal $\a \sub A$ such that $Z = \opn{Zer}_X(\a)$. 
Define the ring $\bar{A} := A / \a$, so $Z \cong \opn{Spec}(\bar{A})$
as topological spaces. Note that $X_{\mrm{max}} \cap Z = \varnothing$.

Assume, for the sake of contradiction, that $Z \neq \varnothing$. 
Then the ring $\bar{A}$ is nonzero, and therefore it has some maximal ideal 
$\bar{\m}$. Let $\m \sub A$ be the preimage of $\bar{\m}$ under the canonical 
surjection $A \to \bar{A}$. Then $\m$ is a 
maximal ideal of $A$, and it satisfies 
$\m \in X_{\mrm{max}} \cap Z$, which is a contradiction. We conclude that 
$Z = \varnothing$. 

It follows that $\bigcup_{i \in I} V_i = X$. Because the 
topological space $X$ is 
quasi-compact, there exists some finite subset $I_0 \sub I$ such that 
$X = \bigcup_{i \in I_0} V_i$. But then 
$X_{\mrm{max}} = \bigcup_{i \in I_0} U_i$.
\end{proof}

\begin{rem} \label{rem:395}
$\opn{MSpec}$ is not a functor on $\cat{Rng}$.
Concretely, given a ring homomorphism $\phi : A \to B$ and a maximal
ideal $\n \sub B$, the prime ideal $\p := \phi^{-1}(\n) \sub A$ is often not 
maximal; see Example \ref{exa:365} below. 

In Section \ref{sec:abst-rings} 
of the paper we will study certain categories of rings on which $\opn{MSpec}$ 
is 
a functor (see Proposition \ref{prop:450}).
\end{rem}

\begin{exa} \label{exa:365}
Consider the polynomial ring $\R[t]$ in the variable $t$ over the 
field $\R$, and its fraction field $\R(t)$. The inclusion homomorphism is
$\phi : \R[t] \to \R(t)$.
The ideal $\n := (0) \sub \R(t)$ is maximal, yet its preimage
$\p := \phi^{-1}(\n) = (0) \sub \R[t]$ is not maximal.
\end{exa}

We end this section with a discussion of the failure of the Hausdorff property 
in typical rings that appear in algebraic geometry. This is a prelude to the 
opposite behavior of {\em BC $\R$-rings}, cf.\ Lemma \ref{lem:272}(1) and 
Theorem \ref{thm:203}.

\begin{exa} \label{exa:362}
Again consider the polynomial ring $A := \R[t]$. 
Let $X := \opn{Spec}(A)$, viewed only as a
topological space with its Zariski topology. (In algebraic geometry the affine
scheme $X$ is called the affine line over $\R$, with notation
$\mbf{A}^{1}_{\R}$.) Like every affine scheme, the topological space $X$ is 
quasi-compact. But it is not Hausdorff, 
because the generic point $\p = (0) \in X$ belongs to every nonempty open 
subset of $X$.

Now we look at the topological subspace $X_{\mrm{max}} := \opn{MSpec}(A)$ of
$X = \opn{Spec}(A)$. The points of $X_{\mrm{max}}$ are the maximal ideals
$\m = (p) \sub A$, where $p = p(t)$ is a monic irreducible polynomial,
necessarily of degree $1$ or $2$.
By Proposition \ref{prop:395} the topological space $X_{\mrm{max}}$ is  
quasi-compact. We claim it is not Hausdorff. Note that the argument used in 
the previous paragraph does not apply here. 

To demonstrate that $X_{\mrm{max}}$ is not Hausdorff, we will prove that
every pair $U, V \sub X_{\mrm{max}}$ of nonempty open subsets satisfies
$U \cap V \neq \varnothing$. To see that this is true, it is enough to look at 
a 
pair of nonempty principal open subsets
$U = \opn{NZer}_{X_{\mrm{max}}}(a)$ and
$V = \opn{NZer}_{X_{\mrm{max}}}(b)$
for $a, b \in A = \R[t]$. Because these open subsets are nonempty, the 
polynomials $a$ and $b$ are nonzero. The product $a \cd b$ is then a nonzero 
polynomial, and it has finitely many irreducible factors. Take an irreducible 
polynomial $p \in A= \R[t]$ that does not divide $a \cd b$. Then the maximal 
ideal $x = \m := (p)$ belongs to $U \cap V$.
\end{exa}

\section{Abstract Function Rings} 
\label{sec:abst-rings}

In this section we study certain categories of rings on which $\opn{MSpec}$ 
is functorial. 

\begin{conv} \label{conv:450}
Throughout this section $\K$ is some fixed base field.
\end{conv}
 
Recall that for a $\K$-ring $A$ we denote the structural 
homomorphism by $\opn{str}_A : \K \to A$. See Definition \ref{dfn:365}.

\begin{dfn} \label{dfn:395}
Let $A$ be a $\K$-ring. 
\begin{enumerate}
\item A maximal ideal $\m \sub A$ is called {\em $\K$-valued} if  the 
structural homomorphism $\opn{str}_{A / \m} : \K \to A / \m$ is bijective. 

\item The $\K$-ring $A$ is said to be a {\em $\K$-valued ring} if all its 
maximal ideals are $\K$-valued.
\end{enumerate}
\end{dfn}

\begin{exa} \label{exa:450}
Assume $\K$ is an algebraically closed field, and let $A$ be a finitely 
generated $\K$-ring. The {\em Hilbert Nullstellensatz} says that $A$ is a 
$\K$-valued ring. 
\end{exa}

\begin{exa} \label{exa:395}
Assume $\K = \R$. Consider the polynomial ring $A := \R[t]$. It is an 
$\R$-ring, 
but it has maximal ideals that are not $\R$-valued, such as $\m := (t^2 + 1)$. 
So $A$ is not an $\R$-valued ring. 

Now let $S \sub A$ be the set of all finite products 
$p_1(t) \cdots p_n(t)$, where $n \geq 0$, and the $p_i(t)$ are monic 
{\em quadratic} irreducible polynomials. Let $B := A_S$, the localization of 
$A$ with respect to $S$. This localization removes all the $\C$-valued maximal 
ideals of $A$. The maximal ideals of $B$, those remaining after the 
localization, are of the form 
$\m = (t - \la)$ for $\la \in \R$. Thus $B$ is an $\R$-valued $\R$-ring.
\end{exa}

\begin{rem} \label{rem:405}
Given a $\K$-ring $A$, define $X := \opn{Spec}(A)$ and 
$X_{\mrm{max}} := \opn{MSpec}(A)$. 
Here is a standard definition from algebraic geometry: the set of {\em 
$\K$-valued points of $X$} is 
$X(\K) := \opn{Hom}_{\cat{Rng} \over \K}(A, \K)$.
There is always an inclusion 
$X(\K) \sub X_{\mrm{max}}$. It is easy to see that the ring $A$ is $\K$-valued 
if and only if $X(\K) = X_{\mrm{max}}$.
\end{rem}

\begin{lem} \label{lem:242}
Let $\phi : A \to B$ be a homomorphism in
$\cat{Rng} \over \K$, and let $\n \sub B$ be a $\K$-valued maximal ideal. Then 
the ideal $\m := \phi^{-1}(\n) \sub A$ is maximal and $\K$-valued.
\end{lem}

\begin{proof}
There is an injective $\K$-ring homomorphism
$\bar{\phi} : A / \m \to B / \n$. By assumption the structural
homomorphism $\K \to B / \n$ is bijective, and this implies that $\bar{\phi}$
is surjective. Thus $\bar{\phi}$ is an isomorphism, 
$A / \m \cong \K$, and $\m$ is maximal.
\end{proof}

\begin{dfn} \label{dfn:450}
The full subcategory of $\cat{Rng} \over \K$ on the $\K$-valued rings is 
denoted by $\cat{Rng} \xover{val} \K$. 
\end{dfn}

\begin{prop} \label{prop:450}  
There is a unique functor
\[ \opn{MSpec} : (\cat{Rng} \xover{val} \K)^{\mrm{op}} \to 
\cat{Top} \]
sending a ring $A \in \cat{Rng}\xover{val} \K$
to the topological space $\opn{MSpec}(A)$, and sending a homomorphism 
$\phi : A \to B$ in $\cat{Rng}\xover{val} \K$
to the continuous map
\[ \tag{$*$} \opn{MSpec}(\phi) : \opn{MSpec}(B) \to \opn{MSpec}(A) , \quad
\n \mapsto \phi^{-1}(\n) . \]
\end{prop}

\begin{proof}
Take a homomorphism $\phi : A \to B$ in
$\cat{Rng}\xover{val} \K$ and some $\n \in \opn{MSpec}(B)$. 
By Lemma \ref{lem:242} the ideal $\phi^{-1}(\n) \sub A$ is 
maximal, so the map of sets $\opn{MSpec}(\phi)$ is well-defined.
As for continuity of $\opn{MSpec}(\phi)$, this is because the map
\[ \opn{Spec}(\phi) : \opn{Spec}(B) \to \opn{Spec}(A) , \quad
\q \mapsto \phi^{-1}(\q) \]
is continuous for the Zariski topologies of $\opn{Spec}(A)$ and 
$\opn{Spec}(B)$, and the spaces $\opn{MSpec}(A)$ and $\opn{MSpec}(B)$
have the respective induced subspace topologies.
We have thus shown that $\opn{MSpec}(\phi)$ is a morphism in 
$\cat{Top}$. 

Trivially $\opn{MSpec}$ respects identity automorphisms.
Finally, regarding compositions: given 
$A \xar{\phi} B \xar{\psi} C$ in $\cat{Rng}\xover{bc} \R$,
formula ($*$) shows that 
\[ \opn{MSpec}(\phi) \circ \opn{MSpec}(\psi) = \opn{MSpec}(\psi \circ \phi) . 
\qedhere \] 
\end{proof}

\begin{dfn} \label{dfn:265}
Let $A$ be a $\K$-ring and $\m \sub A$ a $\K$-valued 
maximal ideal, with canonical surjection $\opn{pr}_{\m} : A \to A / \m$.
Define the {\em evaluation ring homomorphism} to be the $\K$-ring 
homomorphism
\[ \opn{ev}_{\m} := (\opn{str}_{A / \m})^{-1} \circ \opn{pr}_{\m} : 
A \to \K . \]
\end{dfn}

In other words, for all $a \in A$ there is equality 
$\opn{str}_{A / \m}(\opn{ev}_{\m}(a)) = a + \m \in A / \m$. 
In a commutative diagram it looks like this:
\[ \begin{tikzcd} [column sep = 10ex, row sep = 6ex]
\K
\ar[r, "{ \opn{str}_{A} }"]
\ar[dr, "{ \opn{str}_{A / \m} }"', "{\simeq}"]
\ar[rr, bend left = 20, start anchor = north east, end anchor = north west,
, "{\opn{id}_{\K}}", "{\simeq}"']
&
A 
\ar[d, "{ \opn{pr}_{\m} }"]
\ar[r, dashed, "{ \opn{ev}_{\m} }"]
&
\K
\ar[dl, "{ \opn{str}_{A / \m} }", "{\simeq}"']
\\
&
A / \m
\end{tikzcd} \]

\begin{dfn} \label{dfn:452}
Let $X$ be a topological space. We denote by $\opn{F}(X, \K)$ the $\K$-ring 
of 
all functions $a : X \to \K$.
\end{dfn}

Clearly, as $X$ changes we have a functor
\[ \opn{F}(-, \K) : \cat{Top}^{\mrm{op}} \to \cat{Rng} \over \K . \]

\begin{dfn} \label{dfn:270}
Let $A$ be a $\K$-valued $\K$-ring.
We define the {\em double evaluation ring homomorphism} 
\[ \opn{dev}_A : A \to \opn{F} \bigl( \opn{MSpec}(A), \K \bigr) \]
to be the $\K$-ring homomorphism with formula 
$\opn{dev}_A(a)(\m) := \opn{ev}_{\m}(a) \in \K$ for $a \in A$ and 
$\m \in  \opn{MSpec}(A)$. 
\end{dfn}

Another way to express $\opn{dev}_A$ is this: for every  
$a \in A$ and $\m \in \opn{MSpec}(A)$ the element 
$\opn{dev}_A(a)(\m) \in \K$ satisfies
\begin{equation} \label{eqn:270}
\opn{str}_{A / \m} \bigl( \opn{dev}_A(a)(\m) \bigr) = 
a + \m \in A / \m . 
\end{equation}

The name "double evaluation" is used because the element
$\opn{dev}_{A}(a)(\m)$ is a function of two arguments: $a$ and $\m$. 

\begin{prop} \label{prop:453}  
The $\K$-ring homomorphism 
$\opn{dev}_A : A \to \opn{F} \bigl( \opn{MSpec}(A) , \K \bigr)$ 
is functorial in the object $A \in \cat{Rng}\xover{val} \K$. 
\end{prop}

\begin{proof}
Let's use the abbreviations 
$\opn{F} := \opn{F}(-, \K)$ and $\opn{M} := \opn{MSpec}$.
Given a homomorphism $\phi : A \to B$ in $\cat{Rng}\xover{val} \K$, we 
need to prove that the diagram 
\begin{equation} \label{eqn:265}
\begin{tikzcd} [column sep = 10ex, row sep = 6ex]
A
\ar[r, "{\opn{dev}_A}"]
\ar[d, "{\phi}"']
&
(\opn{F} \circ \opn{M})(A)
\ar[d, "{(\opn{F} \circ \opn{M})(\phi)}"]
\\
B
\ar[r, "{\opn{dev}_B}"]
&
(\opn{F} \circ \opn{M})(B)
\end{tikzcd} 
\end{equation}
in $\cat{Rng} \over \K$ is commutative.

For the proof we shall require the commutative diagram of ring isomorphisms 
\begin{equation} \label{eqn:280}
\begin{tikzcd} [column sep = 12ex, row sep = 6ex] 
\K
\ar[rr, bend left = 20, start anchor = north east, end anchor = north west,
, "{\opn{str}_{B / \n}}", "{\simeq}"']
\ar[r, "{\opn{str}_{A / \m}}", "{\simeq}"']
&
A / \m
\ar[r, "{\bar{\phi}}", "{\simeq}"']
&
B / \n
\end{tikzcd}
\end{equation}
where $\bar{\phi}(a + \m) = \phi(a) + \n$ for $a \in A$.

Take an element $a \in A$ and a maximal ideal $\n \in \opn{M}(B)$. 
Let $b := \phi(a) \in B$, 
$a' := \opn{dev}_A(a) \in \opn{F}(\opn{M}(A))$,
$b' := \opn{dev}_B(b) \in \opn{F}(\opn{M}(B))$
and $\m := \opn{M}(\phi)(\n) = \phi^{-1}(\n) \in \opn{M}(A)$. 
Next let $\mu := a'(\m) \in \K$ and $\nu := b'(\n) \in \K$. 
We have 
\[ \bigl( (\opn{F} \circ \opn{M})(\phi) \circ 
\opn{dev}_A \bigr) (a)(\n) = 
\bigl( (\opn{F} \circ \opn{M})(\phi) \bigr) (a')(\n) = 
(a' \circ \opn{M}(\phi)) (\n) = a' (\m) = \mu \]
and 
\[ (\opn{dev}_B \circ \, \phi)(a)(\n) = 
\opn{dev}_B(b)(\n) = b'(\n) = \nu ; \]
and it remains to prove that $\mu = \nu$. 

According to formula (\ref{eqn:270}) we know that 
$\opn{str}_{A / \m}(\mu) = a + \m \in A / \m$ and 
$\opn{str}_{B / \n}(\nu) = b + \n \in A / \n$. 
Since $\bar{\phi}(a + \m) = (b + \n)$, 
the commutative diagram of isomorphisms (\ref{eqn:280}) says that $\mu = \nu$. 
\end{proof}

The next theorem explains how the definitions in this section behave with 
respect to finite base change. This theorem will 
be used in Section \ref{BC-C-rings} of the paper.

\begin{thm} \label{thm:458} 
Let $\mbb{L}$ be a finite extension field of $\K$.
Let $A$ be a $\K$-valued $\K$-ring, and define the $\mbb{L}$-ring 
$B := \mbb{L} \lmsp \ot_{\K} \lmsp A$, with $\K$-ring homomorphism
$\ga : A \to B$, $\ga(a) := 1 \ot a$. Then:
\begin{enumerate}
\item $B$ is an $\mbb{L}$-valued ring.

\item For every $\n \in \opn{MSpec}(B)$ the ideal 
$\m := A \cap \n = \ga^{-1}(\n) \sub A$ is maximal.

\item Let $X := \opn{MSpec}(A)$ and $Y := \opn{MSpec}(B)$.  
Then the map $f : Y \to X$, $f(\n) := \ga^{-1}(\n)$, 
is a homeomorphism . 

\item The diagram 
\[ \begin{tikzcd} [column sep = 10ex, row sep = 6ex]
A
\ar[r, "{\ga}"]
\ar[d, "{\opn{dev}_{A}}"']
&
B
\ar[d, "{\opn{dev}_{B}}"]
\\
\opn{F}(X, \K)
\ar[r, "{\opn{F}(f, \opn{str}_{\mbb{L}})}"]
&
\opn{F}(Y, \mbb{L})
\end{tikzcd} \]
in $\cat{Rng} \over \K$ is commutative.
\end{enumerate}
\end{thm}

The homomorphisms $\opn{dev}_{A}$ and $\opn{dev}_B$ exist because 
the rings $A$ and $B$ are \lb $\K$-valued and $\mbb{L}$-valued, respectively.

\begin{proof} \mbox{}

\smallskip \noindent
(1, 2) Given a maximal ideal $\n \sub B$, define the prime ideal 
$\m := A \cap \n \sub A$ and the rings $L := B / \n$ and 
$K := A / \m$. So $L$ is a field and $K \sub L$ is a subring. 
Since the ring homomorphism 
$\mbb{L} \ot_{\K} K \to L$ is surjective we see that 
$L$ is a finitely generated $K$-module. According to 
\cite[Theorem 10.18]{AK}, $L$ is an integral extension of $K$. Next,  
according to \cite[Lemma 14.1]{AK} the ring $K$ is also a field. Therefore 
the ideal $\m \sub A$ is maximal. Because $A$ is a $\K$-valued ring, it 
follows that $K = A / \m \cong \K$. The surjective ring homomorphism 
$\mbb{L} \ot_{\K} K \to L$ now says that $L \cong \mbb{L}$. 

\medskip \noindent
(3) The map $f : Y \to X$ exists by item (2). 
We shall prove that $f$ is bijective by producing its inverse. 
Given a maximal ideal $\m \sub A$, there is an exact sequence 
\[ 0 \to \m \to A \xar{\opn{ev}_{\m}} \K \to 0 . \]
Applying $\mbb{L} \ot_{\K} (-)$ to this sequence, we obtain the exact sequence 
\[ 0 \to \mbb{L} \ot_{\K} \m \to B \xar{\opn{id} \ot \opn{ev}_{\m}} \mbb{L} \to 
0 . \]
Thus $\n := \mbb{L} \lmsp \ot_{\K} \lmsp \m$ is a maximal ideal of $B$.
In this way we obtain a function $g : X \to Y$, $\m \mapsto \n$. 
An easy calculation shows that the functions $f$ and $g$ are mutual 
inverses; and therefore $f$ is a bijection. 

Let us prove that $f : Y \to X$ is a homeomorphism. 
The map $f$ is continuous because it is the restriction to maximal 
spectra of the continuous map 
$\opn{Spec}(\ga) : \opn{Spec}(B) \to \opn{Spec}(A)$.
It remains to prove that $f$ is an open map. 
Choose a $\K$-basis $(\la_1, \ldots, \la_n)$ of $\mbb{L}$. 
Take a principal open set 
$V = \opn{NZer}_Y(b) \sub Y$ for some $b \in B$. We can expand $b$ into a sum
$b = \sum_{i = 1, \ldots, n} \la_i \cd a_i$ with $a_i \in A$. Then
$V = \bigcap_{i = 1, \ldots, n} \opn{NZer}_Y(a_i)$, 
and hence 
$f(V) = \bigcap_{i = 1, \ldots, n} \opn{NZer}_{X}(a_i)$,
which is an open set of $X$. 

\medskip \noindent
(4) This is an easy calculation.
\end{proof}

\section{Rings of Bounded Continuous Real Valued Functions}
\label{BC-rings}

In the current section we introduce a category of $\R$-rings that plays a 
central role in our paper (and is referred to in the title). These are the {\em 
BC $\R$-rings}, see Definition \ref{dfn:201}. 

The base field $\R$ (and later, in Sections \ref{BC-C-rings}--\ref{sec:inv-c-b},
also the base field $\C$) has two distinct incarnations in this paper: 
sometimes it is just an abstract ring (namely without a topology); and at other 
times it is a topological ring, with its standard norm topology. To minimize 
confusion, {\em by default $\R$ will be considered as an abstract ring}. When 
$\R$ is considered as a topological ring, this will be stated explicitly (like 
in Definition \ref{dfn:200} below). {\em All other rings are abstract rings} 
(with the exception of the Banach rings in Section \ref{sec:inv-c-b}). 
In particular, morphisms in the category 
$\cat{Rng} \over \R$ do not have any continuity condition.

\begin{dfn} \label{dfn:200} 
Let $X$ be a topological space. 
\begin{enumerate}
\item The $\R$-ring consisting of all {\em continuous functions} $a : X \to \R$
is denoted by $\opn{F}_{\mrm{c}}(X, \R)$. 

\item The $\R$-ring  consisting of all {\em bounded 
continuous functions} $a : X \to \R$ is denoted by 
$\opn{F}_{\mrm{bc}}(X, \R)$.
\end{enumerate}
Here continuity and boundedness are with respect to the standard norm on the 
field $\R$. 
\end{dfn}

We repeat: the rings 
$\opn{F}_{\mrm{c}}(X, \R)$ and $\opn{F}_{\mrm{bc}}(X, \R)$ are {\em not 
topologized}. Of course if $X$ is a discrete topological space then 
$\opn{F}_{\mrm{c}}(X, \R) = \opn{F}_{}(X, \R)$, 
and if $X$ is a compact topological space then 
$\opn{F}_{\mrm{bc}}(X, \R) = \opn{F}_{\mrm{c}}(X, \R)$.  
For the empty topological space the ring $\opn{F}(\varnothing, \R)$ is the 
zero ring. 

\begin{rem} \label{rem:325}
The standard notation for the ring $\opn{F}_{\mrm{c}}(X, \R)$ is
$\opn{C}(X)$; but this is not good notation for our paper, for three 
reasons. First, $\opn{C}(X)$ often refers to the ring 
$\opn{F}_{\mrm{c}}(X, \mbb{C})$
of continuous $\C$-valued functions, which will show up in Sections 
\ref{BC-C-rings} and \ref{sec:inv-c-b}, so the notation $\mrm{C}(X)$ would 
cause 
unwanted ambiguity. Second, the ring $\opn{C}(X)$ -- in either of its two 
meanings -- is often viewed as a topological ring, with the sup norm, whereas 
for us it is mostly an abstract ring. And third, in our paper we require more 
refined notions, and hence also more refined notation. 
\end{rem}

It is obvious that 
\begin{equation} \label{eqn:331}
\opn{F}_{\mrm{bc}}(-, \R) , \  
\opn{F}_{\mrm{c}}(-, \R), \ \opn{F}(-, \R) \ : \
\cat{Top}^{\mrm{op}} \ \to \ \cat{Rng} \over \R
\end{equation}
are functors. Furthermore, each of them is a subfunctor of the next 
one. This means that for every topological space $X$ there are 
inclusion of rings 
$\opn{F}_{\mrm{bc}}(X, \R) \sub \opn{F}_{\mrm{c}}(X, \R)
\sub \opn{F}(X, \R)$,
and these inclusions are functorial in $X$. 

\begin{dfn} \label{dfn:201}
An $\R$-ring $A$ is called a {\em BC $\R$-ring} if it is isomorphic,
as an $\R$-ring, to the ring $\opn{F}_{\mrm{bc}}(X, \R)$ for some {\em 
compact} 
topological space $X$. The full subcategory of $\cat{Rng} \over \R$ on the BC 
$\R$-rings is denoted by $\cat{Rng}\xover{bc} \R$.
\end{dfn}

Thus the  category $\cat{Rng}\xover{bc} \R$ is 
the essential image of the functor 
\begin{equation} \label{eqn:330}
\opn{F}_{\mrm{bc}}(-, \R) : (\cat{Top}_{\mrm{cp}})^{\mrm{op}} \to 
\cat{Rng} \over \R . 
\end{equation}
The expression "BC" stands for "bounded continuous". 

\begin{rem} \label{rem:396}
In Theorem \ref{thm:415}  we shall see that for {\em every} topological 
space $X$ the ring $\opn{F}_{\mrm{bc}}(X, \R)$ belongs to 
$\cat{Rng}\xover{bc} \R$. 
This fact is closely related to the existence of the
Stone-\v{C}ech Compactification of $X$.  
\end{rem}

\begin{dfn} \label{dfn:230}
Given a topological space $X$ and a point $x \in X$, let
\[ \opn{ev}_{x} : \opn{F}_{\mrm{bc}}(X, \R) \to \R \]
be the $\R$-ring homomorphism $\opn{ev}_{x}(a) := a(x)$. It is called the 
{\em evaluation at $x$ homomorphism}. 
\end{dfn}

It is clear that $\m := \opn{Ker}(\opn{ev}_{x})$ is an $\R$-valued maximal 
ideal of the ring $A := \opn{F}_{\mrm{bc}}(X, \R)$.

\begin{dfn} \label{dfn:231}
For a topological space $X$ define the map of sets
\[ \opn{refl}^{\mrm{\lmsp alg}}_X : X \to
\opn{MSpec} \bigl( \opn{F}_{\mrm{bc}}(X, \R) \bigr) , \quad
x \mapsto \opn{Ker}(\opn{ev}_{x}) . \]
We call it the {\em algebraic reflection map}. 
\end{dfn}

Here is a general topological analogue of the definition for prime spectra,
see Definition \ref{dfn:460}. 

\begin{dfn} \label{dfn:325}
Given a topological space $X$ and a function $a \in \opn{F}_{\mrm{c}}(X, \R)$,
we define the open subset
\[ \opn{NZer}_X(a) := \{ x \in X \mid a(x) \neq 0 \} \sub X , \]
and call it the {\em principal open set} determined by $a$.
\end{dfn}

The subset $\opn{NZer}_X(a)$ is open in $X$ is because the function
$a : X \to \R$ is continuous for the norm topology of $\R$. 
The relation between Definitions \ref{dfn:230} and \ref{dfn:325} is this:
\begin{equation} \label{eqn:325}
\opn{NZer}_X(a) = \{ x \in X \mid \opn{ev}_{x}(a) \neq 0 \} 
\end{equation}

\begin{prop} \label{prop:320}   
Let $X$ be a topological space, and define the topological space
$\bar{X} := \opn{MSpec} \bigl( \opn{F}_{\mrm{bc}}(X, \R) \bigr)$.
\begin{enumerate}
\item The algebraic reflection map
$\opn{refl}^{\mrm{\lmsp alg}}_X : X \to \bar{X}$ is continuous.

\item The image of $\opn{refl}^{\mrm{\lmsp alg}}_X$ is dense in $\bar{X}$.
\end{enumerate}
\end{prop}

\begin{proof} \mbox{}

\smallskip (1) Let's write $A := \opn{F}_{\mrm{bc}}(X, \R)$
and $h := \opn{refl}^{\mrm{\lmsp alg}}_X$. 
Recall that $\bar{X} = \opn{MSpec}(A)$ has the Zariski topology. Thus, to 
prove continuity of $h$ it suffices to show that for every
principal open set $\bar{U} = \opn{NZer}_{\bar{X}}(a) \sub \bar{X}$, determined 
by an element $a \in A$, the preimage $h^{-1}(\bar{U}) \sub X$ is open. But
$h^{-1}(\bar{U}) = \opn{NZer}_{X}(a)$, and this is open in $X$ 
as explained above. 

\medskip \noindent
(2) Consider a nonempty principal open set 
$\bar{U} = \opn{NZer}_{\bar{X}}(a) \sub \bar{X}$ for some $a \in A$. 
Since $\bar{U} \neq \varnothing$ the element $a \in A$ is nonzero, so as a 
function $a : X \to \R$ must be nonzero. But 
then there must be some point $x \in X$ such that $a(x) \neq 0$. This $x$
satisfies $\opn{refl}^{\mrm{\lmsp alg}}_X(x) \in \bar{U}$. 
\end{proof}

\section{Duality for Compact Topological Spaces}
\label{sec:comp-aff}

In this section we prove Theorem \ref{thm:203}, which asserts that the 
categories $\cat{Top}_{\mrm{cp}}$ and $\cat{Rng}\xover{bc} \R$ are dual 
to each other. Recall that $\cat{Rng}\xover{bc} \R$ is the category of 
BC $\R$-rings, without a topology (Definition \ref{dfn:201}), and 
$\cat{Top}_{\mrm{cp}}$ is the category of compact topological spaces and 
continuous maps (see Convention \ref{conv:452}).

The property of compact topological spaces that is important for us is stated 
in the next theorem. This is a classical result, usually appearing as two 
distinct theorems, as the proof shows. 

\begin{thm}[Separation by Continuous Functions] \label{thm:245}
Let $X$ be a compact topological space, and let $Y_0$ and $Y_1$ be disjoint 
closed subsets of $X$. Let $Z := [0, 1]$, the closed interval in $\R$, with its 
usual norm topology. Then there is a continuous function 
$f : X \to Z$ such that $f(Y_0) \sub \{ 0 \}$ and $f(Y_1) \sub \{ 1 \}$.
\end{thm}

\begin{proof}
According to \cite[Theorem 32.3]{Mn} or \cite[Theorem 5.9]{Ke}
the space $X$ is normal, and according to 
\cite[Theorem 33.1]{Mn} or \cite[Lemma 4.4]{Ke} (the Urysohn Lemma) 
such a function $a$ exists. 
\end{proof}

\begin{lem} \label{lem:230}
Let $X$ be a compact topological space. 
\begin{enumerate}
\item The elements of $\opn{F}_{\mrm{bc}}(X, \R)$ separate the points of $X$.
Namely, given two distinct points $x, y \in X$, there exists a function 
$a \in \opn{F}_{\mrm{bc}}(X, \R)$ such that $a(x) \neq a(y)$.

\item The principal open sets form a basis of the topology of $X$. By this we 
mean that for every open set $U \sub X$ and every point $x \in U$ there exists 
some $a \in \opn{F}_{\mrm{bc}}(X, \R)$ such that 
$x \in \opn{NZer}_X(a) \sub U$. 
\end{enumerate}
\end{lem}

\begin{proof} 
(1) Consider the  disjoint closed subsets $\{ x \}$ and $\{ y \}$ of $X$. 
By the Theorem \ref{thm:245} there is a bounded continuous function 
$a : X \to \R$ such that $a(x) = 1$ and $a(y) = 0$. 

\medskip \noindent
(2) Consider the disjoint closed subsets $\{ x \}$ and 
$Z := X - U$ of $X$. By Theorem \ref{thm:245}  there is a bounded continuous 
function $a : X \to \R$ such that $a(x) = 1$ and $a(z) = 0$ for all $z \in Z$. 
For such $a$ we have $x \in \opn{NZer}_X(a) \sub U$. 
\end{proof}

Given a ring $A$, its maximal spectrum $\opn{MSpec}(A)$ is equipped with the 
Zariski topology (recalled in Section \ref{sec:prlim}).

\begin{lem} \label{lem:232} 
Let $X$ be a compact topological space, and let 
$A := \opn{F}_{\mrm{bc}}(X, \R)$.
\begin{enumerate}
\item Given $\m \in \opn{MSpec}(A)$, there is a point $x \in X$ such 
that $\m = \opn{Ker}(\opn{ev}_{x}) = \opn{refl}^{\mrm{\lmsp alg}}_X(x)$. 

\item The algebraic reflection map 
$\opn{refl}^{\mrm{\lmsp alg}}_X : X \to \opn{MSpec}(A)$
is a homeomorphism.
\end{enumerate}
\end{lem}

\begin{proof} 
These assertions are not new, cf.\ \cite[Exercise 14.26]{AK}. Item (1) 
has a full proof in \cite[Proposition 1.22]{Wa}. For the convenience of the 
reader we are going to provide a proof of item (2). 

Let's write $\bar{X} := \opn{MSpec}(A)$.
Lemma \ref{lem:230}(1) implies that the map 
$\opn{refl}^{\mrm{\lmsp alg}}_X : X \to \bar{X}$ is injective. 
Indeed, if $x, y \in X$ are distinct, and $a \in A$ is such 
that $a(x) = 1$ and $a(y) = 0$, then $a \in \opn{refl}^{\mrm{\lmsp alg}}_X(y)$ 
but 
$a \notin \opn{refl}^{\mrm{\lmsp alg}}_X(x)$, so these maximal ideals are 
distinct. 
Item (1) above shows that the map $\opn{refl}^{\mrm{\lmsp alg}}_X$ is 
surjective.
We see that $\opn{refl}^{\mrm{\lmsp alg}}_X : X \to \bar{X}$ is bijective.

It remains to prove that $\opn{refl}^{\mrm{\lmsp alg}}_X : X \to \bar{X}$ is a 
homeomorphism. Given an element $a \in A$, consider the principal open sets 
$U := \opn{NZer}_X(a) \sub X$ and 
$\bar{U} := \opn{NZer}_{\bar{X}}(a) \sub \bar{X}$. An easy calculation shows 
that 
$\opn{refl}^{\mrm{\lmsp alg}}_X(U) = \bar{U}$. We know (see Section 
\ref{sec:prlim}) 
that the principal open sets $U'$ form a basis of the Zariski topology of 
$\bar{X}$, and according to Lemma \ref{lem:230}(2) the principal open sets 
$U$ form a basis of the given topology of $X$. Therefore 
$\opn{refl}^{\mrm{\lmsp alg}}_X$ is a homeomorphism.
\end{proof}

\begin{lem} \label{lem:251}
If $A$ is a BC $\R$-ring then it is $\R$-valued.
\end{lem}

\begin{proof}
We can assume that $A = \opn{F}_{\mrm{bc}}(X, \R)$ for some compact 
topological 
space $X$. Let $\m$ be a maximal ideal of $A$. 
By Lemma \ref{lem:232} there is a point $x \in X$ such that 
$\m =  \opn{Ker}(\opn{ev}_x)$. But $\opn{ev}_x$ is an $\R$-ring homomorphism 
$\opn{ev}_x : A \to \R$, so it induces an $\R$-ring isomorphism 
$A / \m \iso \R$. 
\end{proof}

\begin{rem} \label{rem:345}
Lemma \ref{lem:251} can be understood as an analogy of the Hilbert 
Nullstellensatz, see Example \ref{exa:450}. 
However a BC $\R$-ring $A$ is almost always infinitely generated as an 
$\R$-ring, and almost never noetherian, so the analogy to classical 
algebraic geometry breaks down fast. 
\end{rem}

\begin{prop} \label{prop:451}    
The maximal spectrum is a functor
\[ \opn{MSpec} : (\cat{Rng}\xover{bc} \R)^{\mrm{op}} \to 
\cat{Top} . \]
Is sends a homomorphism $\phi : A \to B$ in $\cat{Rng}\xover{bc} \R$
to the continuous map 
$\opn{MSpec}(\phi) : \opn{MSpec}(B) \to \opn{MSpec}(A)$,
$\opn{MSpec}(\phi)(\n) = \phi^{-1}(\n)$. 
\end{prop}

\begin{proof}
Use Lemma \ref{lem:251} and Proposition \ref{prop:450}.
\end{proof}

By definition there is a functor 
\[ \opn{F}_{\mrm{bc}}(-, \R) : (\cat{Top}_{\mrm{cp}})^{\mrm{op}} \to 
\cat{Rng}\xover{bc} \R . \]
In view of Proposition \ref{prop:451} there is a composed functor
\[ \opn{MSpec} \circ \opn{F}_{\mrm{bc}}(-, \R) : 
\cat{Top}_{\mrm{cp}} \to \cat{Top} . \]
There is also the inclusion functor
$\opn{Inc} : \cat{Top}_{\mrm{cp}} \to \cat{Top}$.

\begin{lem} \label{lem:258}
The algebraic reflection map
\[ \opn{refl}^{\mrm{\lmsp alg}}_X : X \to
\opn{MSpec} \bigl( \opn{F}_{\mrm{bc}}(X, \R) \bigr) \]
in $\cat{Top}$ is functorial in the object $X \in \cat{Top}_{\mrm{cp}}$.
In other words, as $X$ changes there is a morphism 
\[ \opn{refl}^{\mrm{\lmsp alg}} : \opn{Inc} \to \opn{MSpec} \circ 
\opn{F}_{\mrm{bc}}(-, \R) \]
of functors $\cat{Top}_{\mrm{cp}} \to \cat{Top}$.
\end{lem}

\begin{proof}
We shall use the abbreviations 
$\opn{M} := \opn{MSpec}$ and 
$\opn{F}_{\mrm{bc}} := \opn{F}_{\mrm{bc}}(-, \R)$.
Given a map $f : Y \to X$ in $\cat{Top}_{\mrm{cp}}$, we must prove that 
the diagram
\begin{equation} \label{eqn:260}
\begin{tikzcd} [column sep = 10ex, row sep = 6ex]
Y
\ar[r, "{\opn{refl}^{\mrm{\lmsp alg}}_Y}"]
\ar[d, "{f}"']
&
(\opn{M} \circ \opn{F}_{\mrm{bc}})(Y)
\ar[d, "{(\opn{M} \circ \opn{F}_{\mrm{bc}})(f)}"]
\\
X
\ar[r, "{\opn{refl}^{\mrm{\lmsp alg}}_X}"]
&
(\opn{M} \circ \opn{F}_{\mrm{bc}})(X)
\end{tikzcd} 
\end{equation}
in $\cat{Top}$ is commutative. 

Write 
$A := \opn{F}_{\mrm{bc}}(X)$, $B := \opn{F}_{\mrm{bc}}(Y)$ and 
$\phi := \opn{F}_{\mrm{bc}}(f) : A \to B$.   
Take a point $y \in Y$, and define
$x := f(y) \in X$, $\n := \opn{refl}^{\mrm{\lmsp alg}}_Y(y) \in \opn{M}(B)$ and
$\m := \opn{M}(\phi)(\n) = \phi^{-1}(\n) \in \opn{M}(A)$.
We need to prove that $\m = \opn{refl}^{\mrm{\lmsp alg}}_X(x)$. 
By definition we have 
$\n = \opn{Ker}(\opn{ev}_{y}) \sub B$, and also 
$\phi(a)(y) = a(f(y)) = a(x)$. Now we calculate:
\[ \begin{aligned}
&
\m = \phi^{-1}(\n) = \phi^{-1}(\opn{Ker}(\opn{ev}_{y})) = 
\opn{Ker}(\opn{ev}_{y} \circ \, \phi) = 
\{ a \in A \mid \phi(a)(y) = 0 \}
\\ & \quad 
= \{ a \in A \mid a(x) = 0 \} = 
\opn{Ker}(\opn{ev}_{x}) = \opn{refl}^{\mrm{\lmsp alg}}_X(x) &  
\end{aligned} \]
as required.
\end{proof}

\begin{lem} \label{lem:275} 
Let $X$ be a compact topological space, and define 
$A := \opn{F}_{\mrm{bc}}(X, \R)$.
\begin{enumerate}
\item Take $a \in A$ and $x \in X$. Let 
$\m := \opn{refl}^{\mrm{\lmsp alg}}_X(x) \in \opn{MSpec}(A)$. 
Then there is equality 
$\opn{dev}_A(a)(\m) = a(x)$ in $\R$. 

\item For every $a \in A$ the function 
$\opn{dev}_A(a) : \opn{MSpec}(A) \to \R$ is continuous (\,for the norm 
topology of $\R$) and bounded; i.e.\
$\opn{dev}_A(a) \in  
\opn{F}_{\mrm{bc}} \bigl( \opn{MSpec}(A), \R \bigr)$. 

\item The $\R$-ring homomorphism 
$\opn{dev}_A : A \to \opn{F}_{\mrm{bc}} 
\bigl( \opn{MSpec}(A), \R \bigr) $
is bijective. 
\end{enumerate}
\end{lem}

\begin{proof} (1) Let 
$\mu := \opn{ev}_{\m}(a) = \opn{dev}_A(a)(\m) \in \R$ and 
$\nu := \opn{ev}_{x}(a) = a(x) \in \R$. We need to prove that $\mu = \nu$. 

Recall that $\m = \opn{Ker}(\opn{ev}_x)$, so there is a unique $\R$-ring 
isomorphism $\psi : A / \m \iso \R$ such that the diagram 
\begin{equation} \label{eqn:281}
\begin{tikzcd} [column sep = 10ex, row sep = 6ex]
A
\ar[dr, two heads, "{\opn{ev}_x}"]
\ar[d, two heads, "{\opn{pr_{\m}}}"']
\\
A / \m
\ar[r, tail, two heads, "{\psi}"]
&
\R
\end{tikzcd} 
\end{equation}
in $\cat{Rng} \over \R$ is commutative. 
Then $\psi$ satisfies $\psi(a + \m) = \opn{ev}_x(a) = \nu$.
But by formula (\ref{eqn:270}) we know that 
$\opn{str}_{A / \m}(\mu) = a + \m \in A / \m$, so 
$\psi(a + \m) = (\psi \circ \opn{str}_{A / \m})(\mu) = \mu$. 

\medskip \noindent 
(2) Item (1) says that the diagram 
\begin{equation} \label{eqn:282}
\begin{tikzcd} [column sep = 10ex, row sep = 6ex]
X
\ar[dr, "{a}"']
\ar[r, "{\opn{refl}^{\mrm{\lmsp alg}}_X}", "{\simeq}"']
&
\opn{MSpec}(A)
\ar[d, "{\opn{dev}_A(a)}"]
\\
&
\R
\end{tikzcd} 
\end{equation}
in $\cat{Set}$ is commutative. Therefore 
$\opn{dev}_A(a) = a \circ (\opn{refl}^{\mrm{\lmsp alg}}_X)^{-1}$. 
By Lemma \ref{lem:232}(2) the map $(\opn{refl}^{\mrm{\lmsp alg}}_X)^{-1}$ is 
continuous, and by definition $a$ is bounded continuous. Hence $\opn{dev}_A(a)$ 
is bounded continuous.

\medskip \noindent 
(3) The commutative diagram (\ref{eqn:282}) shows that the function 
$\opn{dev}_A$ is bijective, with inverse 
$\bar{a} \mapsto \opn{refl}^{\mrm{\lmsp alg}}_X \circ \, \bar{a}$. 
\end{proof}

\begin{lem} \label{lem:272}
Let $A \in \cat{Rng}\xover{bc} \R$ and 
$X := \opn{MSpec}(A)$. 
\begin{enumerate}
\item The topological space $X$ is compact.  

\item For every $a \in A$ the function $\opn{dev}_A(a) : X \to \R$ is 
continuous (\,for the norm topology of $\R$) and bounded.

\item The $\R$-ring homomorphism 
$\opn{dev}_A : A \to \opn{F}_{\mrm{bc}}(X, \R)$
is bijective. 
\end{enumerate}
\end{lem}

\begin{proof}
(1) \ltsp We can assume that $A = \opn{F}_{\mrm{bc}}(Y, \R)$ for some compact 
topological space $Y$. By Lemma \ref{lem:232}(2) there is a homeomorphism 
$\opn{refl}^{\mrm{\lmsp alg}}_Y : Y \iso X = \opn{MSpec}(A)$. So $X$ is 
compact. 

\medskip \noindent 
(2, 3) \ltsp Let's introduce the abbreviations 
$\opn{M} := \opn{MSpec}$, $\opn{F} := \opn{F}(-, \R)$ 
and $\opn{F}_{\mrm{bc}} := \opn{F}_{\mrm{bc}}(-, \R)$.
Choose a ring isomorphism $\phi : A \iso B$, where
$B = \opn{F}_{\mrm{bc}}(Y)$ for some compact topological space $Y$. 
By Proposition \ref{prop:451} we know that  
$\opn{M} :  (\cat{Rng}\xover{bc} \R)^{\mrm{op}} \to \cat{Top}$
is a functor. There are also the functors
$\opn{F}_{\mrm{bc}}, \opn{F} : \cat{Top}^{\mrm{op}} \to \cat{Rng} \over \R$, 
and $\opn{F}_{\mrm{bc}} \sub \opn{F}$ is a subfunctor. 
According to Proposition \ref{prop:453} there is a morphism 
$\opn{dev} : \opn{Inc} \to \opn{F} \circ \opn{M}$
of functors 
$\cat{Rng}\xover{bc} \R \to \cat{Rng} \over \R$,
where 
$\opn{Inc} : \cat{Rng}\xover{bc} \R \to \cat{Rng} \over \R$
is the inclusion functor. 
Therefore we have this solid commutative diagram
\begin{equation} \label{eqn:445}
\begin{tikzcd} [column sep = 14ex, row sep = 6ex]
A
\ar[r, "{\phi}", "{\simeq}"']
\ar[d, dashed, "{\opn{dev}_A}"', "{\simeq}"]
\ar[dd, bend right = 35, start anchor = west, end anchor = north west,
"{\opn{dev}_A}"']
&
B
\ar[d, dashed, "{\opn{dev}_B}", "{\simeq}"']
\ar[dd, bend left = 35, start anchor = east, end anchor = north east,
"{\opn{dev}_B}"]
\\
(\opn{F}_{\mrm{bc}} \circ \opn{M})(A)
\ar[r, "{(\opn{F}_{\mrm{bc}} \circ \opn{M})(\phi)}", "{\simeq}"']
\ar[d, "{\opn{inc}}"]
&
(\opn{F}_{\mrm{bc}} \circ \opn{M})(B)
\ar[d, "{\opn{inc}}"']
\\
(\opn{F} \circ \opn{M})(A)
\ar[r, "{(\opn{F} \circ \opn{M})(\phi)}", "{\simeq}"']
&
(\opn{F} \circ \opn{M})(B)
\end{tikzcd} 
\end{equation}
in $\cat{Rng} \over \R$. 
The arrows marked $\opn{inc}$ are inclusions of rings. 
Lemma \ref{lem:275} says that the homomorphism $\opn{dev}_B$ on the 
dashed 
arrow 
going down from $B$ exists, and it is an isomorphism. 
An easy diagram chase shows  the homomorphism $\opn{dev}_A$ on the dashed 
arrow going down from $A$ also exists, it is an isomorphism, and the whole 
diagram (\ref{eqn:445}) is commutative. 
\end{proof}

\begin{rem} \label{rem:406}
The continuity in Lemmas \ref{lem:275}(2) and \ref{lem:272}(2) seems to be 
quite rare. When $A$ is an arbitrary $\R$-valued $\R$-ring, there is no reason 
for the functions $\opn{dev}_A(a) : \opn{MSpec}(A) \to \R$, with $a \in 
A$, to be continuous (for the Zariski topology on $\opn{MSpec}(A)$ and the 
norm topology on $\R$). Below is a counterexample.
\end{rem}

\begin{exa} \label{exa:465}
Take the $\R$-ring $B$ from Example \ref{exa:395}. We showed there that 
$B$ is an $\R$-valued ring. The element $t \in B$ gives a bijection 
of sets $\opn{dev}_B (t) : \opn{MSpec}(B) \to \R$,
but the map $\opn{dev}_B (t)$ is not continuous. 
\end{exa}

\begin{thm}[Duality] \label{thm:203}
The functor 
\[ \opn{F}_{\mrm{bc}}(-, \R) : (\cat{Top}_{\mrm{cp}})^{\mrm{op}} \to 
\cat{Rng}\xover{bc} \R \]
is an equivalence of categories, with quasi-inverse $\opn{MSpec}$.
\end{thm}

\begin{proof}
Let's use the abbreviations $\opn{M} := \opn{MSpec}$ and
$\opn{F}_{\mrm{bc}} := \opn{F}_{\mrm{bc}}(-, \R)$.
Lemma \ref{lem:258} gives a functorial map 
$\opn{refl}^{\mrm{\lmsp alg}}_X : X \to (\opn{M} \circ \opn{F}_{\mrm{bc}})(X)$
in $\cat{Top}$ for every $X \in \cat{Top}_{\mrm{cp}}$, and by Lemma 
\ref{lem:232}(2) this is a homeomorphism, i.e.\ an isomorphism in 
$\cat{Top}$. 

In the reverse direction, the functor
$\opn{M} : (\cat{Rng}\xover{bc} \R)^{\mrm{op}} \to 
\cat{Top}_{\mrm{cp}}$ 
exists by Proposition \ref{prop:451}  and Lemma \ref{lem:272}(1). 
For every $A \in \cat{Rng}\xover{bc} \R$,
Lemma \ref{lem:272}(3) gives a ring isomorphism
$\opn{dev}_A : A \iso (\opn{F}_{\mrm{bc}} \circ \opn{M})(A)$;
and by Proposition \ref{prop:453} this isomorphism is functorial in $A$.
\end{proof}

\section{The \texorpdfstring{Stone-\v{C}ech Compactification}{SCC}}
\label{sec:SCC}

There are several distinct definitions of Stone-\v{C}ech Compactification (SCC) 
in the literature. We prefer Definition \ref{dfn:131} below, which is the most 
general; it is stated, with a proof of existence, in 
\cite[Lemma tag = 
\texttt{\href{https://stacks.math.columbia.edu/tag/0908}{0908}}]{SP}.
Weaker definitions and results can be 
found in the books \cite{GJ}, \cite{Wa} and \cite{Ke}. See Remark 
\ref{rem:285} for a brief discussion. 

The uniqueness of SCC, as it is defined in Definition \ref{dfn:131}, is 
trivial, and so is its functoriality (Proposition \ref{prop:285}). The 
interesting aspect is proving existence. In this section we give an algebraic 
proof of the existence of the SCC (Theorem \ref{thm:285}). 
We also prove Theorem \ref{thm:415}  regarding the functor 
$\opn{F}_{\mrm{bc}}(-, \R)$.

Recall that in this paper a topological space $X$ is called compact if it is
Hausdorff and quasi-compact. The category of topological spaces and continuous 
maps is denoted by $\cat{Top}$, and its full subcategory on the compact spaces 
is denoted by $\cat{Top}_{\mrm{cp}}$.

\begin{dfn} \label{dfn:131}
Let $X$ be a topological space. A {\em Stone-\v{C}ech Compactification of $X$}
is a pair $(\bar{X}, \opn{c}_X)$, consisting of a compact topological space
$\bar{X}$ and a continuous map $\opn{c}_X : X \to  \bar{X}$, such that
the following universal property holds:
\begin{itemize}
\rmitem{C} Given a continuous map $f : X \to Y$, where
the target $Y$ is a compact topological space, there is a unique continuous map
$\bar{f} : \bar{X} \to Y$ such that 
$f = \bar{f} \circ \opn{c}_X$.
\end{itemize}
\end{dfn}

The commutative diagram below, in the category $\cat{Top}$, illustrates 
the universal property (C) in Definition \ref{dfn:131}. 
\begin{equation} \label{eqn:135}
\begin{tikzcd} [column sep = 10ex, row sep = 6ex]
X
\ar[r, "{\opn{c}_X}"]
\ar[dr, "{f}"']
&
\bar{X}
\ar[d, "{\bar{f}}"]
\\
&
Y
\end{tikzcd} 
\end{equation}

As we mentioned above, it is clear that if a Stone-\v{C}ech Compactification 
$(\bar{X}, \opn{c}_X)$ of a space $X$ exists, then it is unique up to a unique 
isomorphism. The next proposition is also immediate from the definition. 

\begin{prop} \label{prop:285}
Assume that every topological space $X$ admits a Stone-\v{C}ech 
Compactification $(\bar{X}, \opn{c}_X)$. Then:
\begin{enumerate}
\item There is a functor 
$\opn{SCC} : \cat{Top} \to \cat{Top}_{\mrm{cp}}$,
sending an object $X \in \cat{Top}$ to its Stone-\v{C}ech Compactification
$\opn{SCC}(X) := \bar{X}$, and sending a map 
$f : X \to Y$ in $\cat{Top}$ to the unique map
$\opn{SCC}(f) : \opn{SCC}(X) \to \opn{SCC}(Y)$
in $\cat{Top}_{\mrm{cp}}$ such that the diagram 
\[ \begin{tikzcd} [column sep = 10ex, row sep = 6ex]
X
\ar[d, "{\opn{c}_X}"']
\ar[r, "{f}"]
&
Y
\ar[d, "{\opn{c}_Y}"]
\\
\opn{SCC}(X) = \bar{X}
\ar[r, "{\opn{SCC}(f) }"]
&
\opn{SCC}(Y) = \bar{Y}
\end{tikzcd} \]
in $\cat{Top}$ is commutative. 

\item A topological space $X$ is compact if and only if 
$\opn{c}_X : X \to \bar{X}$ is a homeomorphism.
\end{enumerate}
\end{prop}

Recall that for a topological space $X$ we denote by 
$\opn{F}_{\mrm{bc}}(X, \R)$ the ring of bounded continuous functions
$a : X \to \R$, where $\R$ has its usual norm topology.
As $X$ changes this becomes a functor 
\begin{equation} \label{eqn:353}
\opn{F}_{\mrm{bc}}(-, \R) : \cat{Top}^{\mrm{op}} \to \cat{Rng} \over \R . 
\end{equation}

The next definition is very similar to material in \cite[page 152]{Ke}.

\begin{dfn} \label{dfn:295}
Let $X$ be some nonempty topological space, and let 
$A := \opn{F}_{\mrm{bc}}(X, \R)$. 
\begin{enumerate}
\item Put on $\R$ its usual norm topology. For each $a \in A$ let $Z_{a}$ 
be the smallest closed interval in $\R$ containing the set $a(X)$. 
Give $Z_a$ the induced subspace topology from $\R$. 

\item Define the topological space 
$\til{X}^{\mrm{top}} := \prod_{a \in A} Z_a$, 
with the product topology.

\item Define the  map of sets 
$\wtil{\opn{refl}}^{\lmsp \mrm{top}}_X : X \to \til{X}^{\mrm{top}}$ to be 
$\wtil{\opn{refl}}^{\mrm{\lmsp top}}_X(x) := \{ a(x) \}_{a \in A}$.

\item Define the topological subspace 
$\bar{X}^{\mrm{top}} \sub \til{X}^{\mrm{top}}$
to be the closure of 
$\wtil{\opn{refl}}^{\lmsp \mrm{top}}_X(X)$, 
with its induced subspace topology. 

\item Let $\opn{refl}^{\mrm{top}}_X : X \to \bar{X}^{\mrm{top}}$
be the induced map of sets. We call it the {\em topological reflection 
map}. 
\end{enumerate}
\end{dfn}

Observe the similarity between the topological reflection map 
$\opn{refl}^{\mrm{\lmsp top}}_X$ defined above and the algebraic reflection map 
$\opn{refl}^{\mrm{\lmsp alg}}_X$ from Definition \ref{dfn:231}. 

\begin{lem} \label{lem:285}
Let $X$ be a nonempty topological space.
\begin{enumerate}
\item The topological space $\bar{X}^{\mrm{top}}$ is compact.

\item The topological reflection map 
$\opn{refl}^{\mrm{top}}_X : X \to \bar{X}^{\mrm{top}}$ is continuous.

\item The subset $\opn{refl}^{\mrm{top}}_X(X) \sub \bar{X}^{\mrm{top}}$ is 
dense.
\end{enumerate}
\end{lem}

\begin{proof} 
(1) Since $\til{X}^{\mrm{top}}$ is a product of compact topological spaces, it 
is compact by the Tychonoff Theorem (see \cite[Theorem 5.13]{Ke}). 
Therefore the closed subspace $\bar{X}^{\mrm{top}} \sub \til{X}^{\mrm{top}}$
is also compact. 

\medskip \noindent
(2) For every $a \in A$ let 
$p_a : \til{X}^{\mrm{top}} \to Z_a$
be the projection on the factor indexed by $a$.
Then 
$p_a \circ \wtil{\opn{refl}}^{\lmsp \mrm{top}}_X = a : X \to \R$
is continuous. By the universal property of the product topology 
the map
$\wtil{\opn{refl}}^{\lmsp \mrm{top}}_X : X \to \til{X}^{\mrm{top}}$
is continuous. But $\bar{X}^{\mrm{top}}$ has the subspace topology induced from 
$\til{X}^{\mrm{top}}$, and therefore the map 
$\opn{refl}^{\lmsp \mrm{top}}_X : X \to \bar{X}^{\mrm{top}}$ is continuous.

\medskip \noindent
(3) This immediate from the definition of $\bar{X}^{\mrm{top}}$.
\end{proof}

Here is the key lemma of this section. 

\begin{lem} \label{lem:286}
For every nonempty topological space $X$, the $\R$-ring homomorphism 
\[  \opn{F}_{\mrm{bc}}(\opn{refl}^{\mrm{top}}_X, \R) : 
\opn{F}_{\mrm{bc}}(\bar{X}^{\mrm{top}}, \R) \to \opn{F}_{\mrm{bc}}(X, \R) \]
is an isomorphism. 
\end{lem}

The lemma makes sense because 
$\opn{refl}^{\mrm{top}}_X : X \to \bar{X}^{\mrm{top}}$ is a map in 
$\cat{Top}$, and $\opn{F}_{\mrm{bc}}(-, \R)$ is a functor, see formula 
(\ref{eqn:353}). See Remark \ref{rem:285} regarding the similarity between 
proof 
of this lemma and other proofs of existence of SCC.

\begin{proof} 
We shall use the abbreviations 
$A := \opn{F}_{\mrm{bc}}(X, \R)$,
$\bar{A} := \opn{F}_{\mrm{bc}}(\bar{X}^{\mrm{top}}, \R)$
and 
$\phi := \opn{F}_{\mrm{bc}}(\opn{dev}^{\mrm{top}}_X, \R) : \bar{A} \to A$.
We need to prove that the $\R$-ring homomorphism $\phi : \bar{A} \to A$ is 
bijective. 

First we'll prove that $\phi$ is injective. Suppose 
$\bar{a}_1, \bar{a}_2 \in \bar{A}$ satisfy 
$\phi(\bar{a}_1) = \phi(\bar{a}_2)$ in $\R$. 
This means that their restrictions to the subspace 
$Y := \opn{refl}^{\mrm{top}}_X(X)$ of $\bar{X}^{\mrm{top}}$ satisfy 
$\bar{a}_1|_{Y} = \bar{a}_2|_{Y}$
as continuous functions $Y \to \R$, where $\R$ if given its standard norm 
topology. By Lemma \ref{lem:286} the space $Y$ is dense in 
$\bar{X}^{\mrm{top}}$, and 
therefore $\bar{a}_1 = \bar{a}_2$.  

Now we will prove that $\phi$ is surjective.
Take some function $a \in A$. Let  
$p_a : \til{X}^{\mrm{top}} \to Z_a$ 
be the projection on the factor indexed by $a$.
Consider the diagram 
\begin{equation} \label{eqn:289}
\begin{tikzcd} [column sep = 10ex, row sep = 6ex]
X
\ar[r, "{\opn{refl}^{\mrm{top}}_X}"]
\ar[dr, "{a}"']
&
\bar{X}^{\mrm{top}}
\ar[r, "{f}"]
\ar[d, "{\bar{a}}"]
&
\til{X}^{\mrm{top}}
\ar[d, "{p_a}"]
\\
&
\R
&
Z_a
\ar[l, "{g}"']
\end{tikzcd} 
\end{equation}
in $\cat{Top}$, where $\R$ has its standard norm topology, $f$ and $g$ are 
the inclusions, and $\bar{a} := g \circ p_a \circ f$. 
The square subdiagram is obviously commutative. 
The definition of $\opn{refl}^{\mrm{top}}_X$ directly implies that the two 
outer 
paths from $X$ to $\R$ are equal. Hence the whole diagram is commutative.
We see that $\bar{a} \circ \opn{refl}^{\mrm{top}}_X = a$. But 
$\bar{a} \circ \opn{refl}^{\mrm{top}}_X = \phi(\bar{a})$. 
\end{proof}

Recall from Definition \ref{dfn:201} that a BC $\R$-ring is an $\R$-ring $A$ 
that is isomorphic to  $\opn{F}_{\mrm{bc}}(X, \R)$ for some compact 
topological 
space $X$. The next theorem shows that these rings occur in much greater
generality.

\begin{thm} \label{thm:415}  For an arbitrary topological space $X$, the ring
$\opn{F}_{\mrm{bc}}(X, \R)$ is a BC $\R$-ring.
\end{thm}

\begin{proof}
We may assume $X$ is a nonempty topological space. 
According to Lemma \ref{lem:286} there is an $\R$-ring isomorphism 
$\opn{F}_{\mrm{bc}}(\bar{X}^{\mrm{top}}, \R) \cong 
\opn{F}_{\mrm{bc}}(X, \R)$,
and by Lemma \ref{lem:285}(1) the topological space $\bar{X}^{\mrm{top}}$ is 
compact. Therefore $\opn{F}_{\mrm{bc}}(X, \R)$ is a BC $\R$-ring. 
\end{proof}

The category $\cat{Rng}\xover{bc} \R$
is the full subcategory of $\cat{Rng} \over \R$ on the BC rings.
Proposition \ref{prop:451} says that there is a functor 
\begin{equation} \label{eqn:335}
\opn{MSpec} : (\cat{Rng}\xover{bc} \R)^{\mrm{op}} \to \cat{Top} .
\end{equation}
According to Theorem \ref{thm:203} the functor 
\begin{equation} \label{eqn:300}
\opn{F}_{\mrm{bc}}(-, \R) : (\cat{Top}_{\mrm{cp}})^{\mrm{op}} \to 
\cat{Rng}\xover{bc} \R 
\end{equation}
is an equivalence of categories, with quasi-inverse $\opn{MSpec}$.
An immediate consequence of Theorem \ref{thm:415} is: 

\begin{cor} \label{cor:415}
The functor $\opn{F}_{\mrm{bc}}(-, \R)$ from formula (\ref{eqn:353})
factors through, and the functor $\opn{F}_{\mrm{bc}}(-, \R)$ from formula 
(\ref{eqn:300}) extends to, a functor 
\[ \opn{F}_{\mrm{bc}}(-, \R) : \cat{Top}^{\mrm{op}} \to 
\cat{Rng}\xover{bc} \R . \]
\end{cor}

\begin{cor} \label{cor:456}
Let $A$ be an $\R$-ring, with $X := \opn{MSpec}(A)$.
The following two conditions are equivalent:
\begin{itemize}
\rmitem{i} $A$ is a BC $\R$-ring.

\rmitem{ii} 
$A$ is an $\R$-valued ring, for every $a \in A$ the function 
$\opn{dev}_A(a) : X \to \R$ is continuous, and the double evaluation
$\R$-ring homomorphism 
$\opn{dev}_A : A \to \opn{F}_{\mrm{c}}(X, \R)$
is bijective. 
\end{itemize}
\end{cor}

\begin{proof} \mbox{}

\smallskip \noindent
(i) $\Rightarrow$ (ii): 
By Lemma \ref{lem:272}(1) the topological space $X$ is compact, and therefore
$\opn{F}_{\mrm{bc}}(X, \R) = \opn{F}_{\mrm{c}}(X, \R)$.
Lemma \ref{lem:251} says that $A$ is an $\R$-valued ring.  
According to Lemma \ref{lem:272}(2), for every $a \in A$ the function 
$\opn{dev}_A(a) : X \to \R$ is 
continuous. Finally, by Lemma \ref{lem:272}(3) the ring homomorphism 
$\opn{dev}_A : A \to \opn{F}_{\mrm{bc}}(X, \R) = 
\opn{F}_{\mrm{c}}(X, \R)$
is bijective.

\medskip \noindent 
(ii) $\Rightarrow$ (i): Proposition \ref{prop:395} says that the topological 
space $X$ is quasi-compact. This implies that 
$\opn{F}_{\mrm{bc}}(X, \R) = \opn{F}_{\mrm{c}}(X, \R)$,
and hence there is an isomorphism
$\opn{dev}_A : A \iso \opn{F}_{\mrm{bc}}(X, \R)$ of $\R$-rings. 
According to Theorem \ref{thm:415}, $A$ is a BC $\R$-ring. 
\end{proof}

\begin{exa} \label{exa:474}
Take the $\R$-ring $B$ from Examples \ref{exa:395} and \ref{exa:465}. 
It is an $\R$-valued ring. Yet for the element $t \in B$ the map 
$\opn{dev}_B (t) : \opn{MSpec}(B) \to \R$ is not continuous. Therefore $B$ is 
not a BC $\R$-ring.  
\end{exa}

Let us denote by 
$\opn{Inc} : \cat{Top}_{\mrm{cp}} \to \cat{Top}$
the inclusion functor. 
The algebraic reflection map
\begin{equation} \label{eqn:354}
\opn{refl}^{\mrm{\lmsp alg}}_X : X \to 
\opn{MSpec} \bigl( \opn{F}_{\mrm{bc}}(X, \R) \bigr) , 
\end{equation}
for an arbitrary topological space $X$, was introduced in Definition 
\ref{dfn:231}. By Lemma \ref{lem:258}, as $X$ moves in 
$\cat{Top}_{\mrm{cp}}$, this becomes a morphism
\begin{equation} \label{eqn:336}
\opn{refl}^{\mrm{\lmsp alg}} : \opn{Inc} \to \opn{MSpec} \circ 
\opn{F}_{\mrm{bc}}(-, \R)
\end{equation}
of functors $\cat{Top}_{\mrm{cp}} \to \cat{Top}$.
There is also the identity functor 
$\opn{Id} : \cat{Top} \to \cat{Top}$.
By Corollary \ref{cor:415} and Proposition \ref{prop:451} there is 
a functor 
\begin{equation} \label{eqn:355}
\opn{MSpec} \circ \opn{F}_{\mrm{bc}}(-, \R) : \cat{Top} \to \cat{Top} .
\end{equation}

\begin{lem} \label{lem:415}
The algebraic reflection map 
$\opn{refl}^{\mrm{\lmsp alg}}_X$ from formula (\ref{eqn:354}) is functorial in 
$X \in \cat{Top}$. Therefore the morphism of functors 
$\opn{refl}^{\mrm{\lmsp alg}}$ in 
formula (\ref{eqn:336}) extends to a morphism
\[ \opn{refl}^{\mrm{\lmsp alg}} : \opn{Id} \to \opn{MSpec} \circ 
\opn{F}_{\mrm{bc}}(-, \R) \]
of functors $\cat{Top} \to \cat{Top}$. 
\end{lem}

\begin{proof} 
Given a map $f : Y \to X$ in $\cat{Top}$, we must prove that the diagram
\[ \begin{tikzcd} [column sep = 10ex, row sep = 6ex]
Y
\ar[r, "{\opn{refl}^{\mrm{\lmsp alg}}_Y}"]
\ar[d, "{f}"']
&
(\opn{M} \circ \opn{F}_{\mrm{bc}})(Y)
\ar[d, "{(\opn{M} \circ \opn{F}_{\mrm{bc}})(f)}"]
\\
X
\ar[r, "{\opn{refl}^{\mrm{\lmsp alg}}_X}"]
&
(\opn{M} \circ \opn{F}_{\mrm{bc}})(X)
\end{tikzcd} \]
in $\cat{Top}$ is commutative. The proof of Lemma \ref{lem:258} works here 
without any modification, except that now we know that 
$\opn{M} \circ \opn{F}_{\mrm{bc}}$ is 
a functor from $\cat{Top}$ to itself (i.e.\ $X$ and $Y$ do not need to be 
compact). 
\end{proof}

The functor $\opn{F}_{\mrm{bc}}(-, \R)$ in Corollary \ref{cor:415}
is not an equivalence, but it does provide an algebraic construction 
of the SCC. 

\begin{thm}[Algebraic SCC] \label{thm:285}
Given a topological space $X$, consider the topological space 
$\bar{X}^{\mrm{alg}} := \opn{MSpec} \bigl( \opn{F}_{\mrm{bc}}(X, \R) \bigr)$
and the algebraic reflection map
$\opn{refl}^{\mrm{\lmsp alg}}_X : X \to \bar{X}^{\mrm{alg}}$ 
from Definition \ref{dfn:231}. Then the pair 
$\big( \bar{X}^{\mrm{alg}}, \opn{refl}^{\mrm{\lmsp alg}}_X \bigr)$ 
is a Stone-\v{C}ech Compactification of $X$. 
\end{thm}

\begin{proof}
We will use the abbreviations 
$\opn{M} := \opn{MSpec}$ and $\opn{F}_{\mrm{bc}} := \opn{F}_{\mrm{bc}}(-, \R)$.
We may assume $X$ is nonempty. 

By Proposition \ref{prop:320}(1) the map
$\opn{refl}^{\mrm{\lmsp alg}}_X : X \to \bar{X}^{\mrm{alg}}$ 
is continuous. By Theorem \ref{thm:415}  the ring
$A := \opn{F}_{\mrm{bc}}(X)$
belongs to $\cat{Rng}\xover{bc} \R$, and therefore by Lemma 
\ref{lem:272}(1) the topological space 
$\bar{X}^{\mrm{alg}} = \opn{M}(A)$ is compact. 

We need to verify that the pair 
$\bigl( \bar{X}^{\mrm{alg}}, \opn{refl}^{\mrm{\lmsp alg}}_X \bigr)$ has the 
universal property (C) from Definition \ref{dfn:131}. 
Suppose $f : X \to Y$ is a map in $\cat{Top}$ with compact target $Y$. 
Due to Proposition \ref{prop:450} and Corollary \ref{cor:415} the map 
\[ (\opn{M} \circ \opn{F}_{\mrm{bc}})(f) : 
(\opn{M} \circ \opn{F}_{\mrm{bc}})(X) \to
(\opn{M} \circ \opn{F}_{\mrm{bc}})(Y) \]
exists in $\cat{Top}$, see formula (\ref{eqn:355}). 
Consider the following solid diagram in $\cat{Top}$. 
\[ \begin{tikzcd} [column sep = 10ex, row sep = 6ex]
X
\ar[r, "{\opn{refl}^{\mrm{\lmsp alg}}_X}"]
\ar[d, "{f}"']
&
\bar{X}^{\mrm{alg}} = (\opn{M} \circ \opn{F}_{\mrm{bc}})(X)
\ar[d, "{(\opn{M} \circ \opn{F}_{\mrm{bc}})(f)}"]
\ar[dl, dashed, "{\bar{f}}"']
\\
Y
\ar[r, "{\opn{refl}^{\mrm{\lmsp alg}}_Y}" near end]
&
(\opn{M} \circ \opn{F}_{\mrm{bc}})(Y)
\end{tikzcd} \]
By Lemma \ref{lem:415} this is a commutative diagram, and by Lemma
\ref{lem:232}(2) the map $\opn{refl}^{\mrm{\lmsp alg}}_Y$ is an isomorphism. 
The 
map
\[ \bar{f} := (\opn{refl}^{\mrm{\lmsp alg}}_Y)^{-1} \circ (\opn{M} \circ 
\opn{F}_{\mrm{bc}})(f) : \bar{X}^{\mrm{alg}} \to  Y \]
in $\cat{Top}$ satisfies $\bar{f} \circ \opn{refl}^{\mrm{\lmsp alg}}_X = f$. 
Such a map $\bar{f}$ is unique because $\opn{refl}^{\mrm{\lmsp alg}}_X(X)$ is 
dense in $\bar{X}^{\mrm{alg}}$, see Proposition \ref{prop:320}(2), and $Y$ is 
Hausdorff. 
\end{proof}

\begin{rem} \label{rem:301}
It is possible to give a more abstract (but more complicated) proof that 
$\bigl( \bar{X}^{\mrm{alg}}, \opn{refl}^{\mrm{\lmsp alg}}_X \bigr)$ has the 
universal property (C) from Definition \ref{dfn:131}.
One first shows that the pair 
$(\opn{M} \circ \opn{F}_{\mrm{bc}}, \opn{refl}^{\mrm{\lmsp alg}})$
is a {\em projector}, also called an {\em idempotent pointed functor} and 
an {\em idempotent monad}. 
The idempotence is proved just like in the proof of Theorem \ref{thm:285}.
Then the categorical properties of projectors 
can be used. See \cite[Section 4.1]{KS} or \cite[Section 2]{VY} for details on 
this abstract approach.
\end{rem}

\begin{rem} \label{rem:285}
The existence of an SCC in the strongest sense, namely the one in Definition
\ref{dfn:131}, is proved in 
\cite[Lemma tag = 
\texttt{\href{https://stacks.math.columbia.edu/tag/0908}{0908}}]{SP}.
Textbooks on topology, such as \cite{GJ}, \cite{Wa} and \cite{Ke},
provide proofs of special cases only; for example \cite[Theorem 5.24]{Ke} only 
considers the SCC of a Tychonoff space. 

All proofs we are aware of, namely those mentioned above, rely on variations of 
the same idea; in our paper this idea occurs in the proof of the Lemma 
\ref{lem:286}.
\end{rem}

\section{Real Banach Rings}
\label{real-B}

In this section we study {\em Banach$^{\Ast}$ $\R$-rings}, better known as {\em 
commutative unital $\mrm{C}^{\Ast}$ $\R$-algebras}. 
In Corollary \ref{cor:485} we prove that the forgetful functor is an 
equivalence from the category of Banach$^{\Ast}$ $\R$-rings to the category of 
BC $\R$-rings. 
Our source for the material here is the book \cite{Go}. 

By a {\em norm} on an $\R$-module $M$ we mean a function
$\norm{-} : M \to \R$ satisfying the following conditions for all $m, n \in M$ 
and $\la \in \R$: 
nonnegativity, $\norm{m} \geq 0$;
nondegeneracy, $\norm{m} = 0$ iff $m = 0$;
the triangle inequality, 
$\norm{m + n} \leq \norm{m} + \norm{n}$; and the homothety equality,
$\norm{\la \cd m} = \abs{\la} \cd \norm{m}$.
The norm induces a metric on $M$ in the obvious way. 
A normed $\R$-module $M$ is called a {\em Banach $\R$-module} if $M$ is 
complete as a metric space. Of course all textbooks will say "$\R$-vector 
space" whenever we say "$\R$-module", and "$\R$-algebra" whenever we say 
"$\R$-ring".

\begin{dfn} \label{dfn:485}
A {\em commutative Banach $\R$-ring} 
is a commutative $\R$-ring $A$, equipped with a norm 
$\norm{-} : A \to \R$, making $A$ into a 
Banach $\R$-module. Moreover, the norm must satisfy these extra conditions: 
$\norm{a \cd b} \leq \norm{a} \cd \norm{b}$ for all $a, b \in A$, and 
$\norm{1_A} = 1$, where $1_A \in A$ is the unit element.
\end{dfn}

\begin{dfn} \label{dfn:486}
A {\em commutative Banach$^{\Ast}$ $\R$-ring} 
is a commutative Banach $\R$-ring $A$, such that for every $a \in A$ 
there is equality $\norm{a^2} = \norm{a}^2$, 
and the element $1 + a^2$ is invertible in $A$. 
\end{dfn}

\begin{dfn} \label{dfn:505} 
Let $A$ and $B$ be commutative Banach $\R$-rings. A {\em Banach $\R$-ring 
homomorphism} is an $\R$-ring homomorphism $\phi : A \to B$ such that 
$\norm{\phi(a)} \leq \norm{a}$ for every element $a \in A$.

The category of commutative Banach $\R$-rings, with Banach $\R$-ring 
homomorphisms, is denoted by $\cat{BaRng} \over \R$. 
The full subcategory of $\cat{BaRng} \over \R$ on the Banach$^{\Ast}$ 
$\R$-rings is denoted by $\cat{BaRng}^{\Ast} \over \R$.
\end{dfn}

\begin{rem} \label{rem:485}
Definition \ref{dfn:486} is a slight modification of the definition of a 
commutative real {\em $\mrm{C}^{\Ast}$ algebra} from \cite[Chapter 8]{Go}. The 
definition in \cite{Go} allows the ring $A$ to have an involution $\ga$; here 
the involution $\ga$ is trivial. We think that the case of a nontrivial 
involution, and the corresponding Arens-Kaplansky Theorem 
(see \cite[Chapter 12]{Go}), should be treated as a {\em noncommutative} case.
\end{rem}

\begin{exa} \label{exa:485}
Suppose $X$ is a compact topological space. Then the ring 
$A := \opn{F}_{\mrm{c}}(X, \R)$, endowed with the $\opn{sup}$ norm, is a 
Banach$^{\Ast}$ $\R$-ring.
\end{exa}

The key classical result on Banach$^{\Ast}$ $\R$-rings is the next one.
It is part of the more general Arens-Kaplansky Theorem, see 
\cite[Theorem 12.5]{Go}.

\begin{thm}[{\cite[Theorem 11.5]{Go}}] \label{thm:485}
Let $A$ be a commutative Banach$^{\Ast}$ $\R$-ring. 
\begin{enumerate}
\item The set $X$ of complex characters of $A$, suitably topologized, is 
compact.  

\item The Gelfand transform $A \to \opn{F}_{\mrm{c}}(X, \R)$
is an isomorphism of Banach$^{\Ast}$ $\R$-rings. 
\end{enumerate}
\end{thm}

\begin{cor} \label{cor:486}
Let $A$ be a commutative Banach$^{\Ast}$ $\R$-ring. Then, after forgetting the 
norm, $A$ is a BC $\R$-ring.  
\end{cor}

\begin{proof}
This is immediate from Theorem \ref{thm:485}. 
\end{proof}

\begin{thm} \label{thm:486} 
Every BC $\R$-ring $A$ admits a unique norm $\norm{-}$, called the 
{\em canonical norm}, satisfying the three conditions below. 
\begin{itemize}
\rmitem{i} Given a homomorphism $\phi : A \to B$ in 
$\cat{Rng} \xover{bc} \R$, and an element $a \in A$, the inequality
$\norm{\phi(a)} \leq \norm{a}$ holds. 

\rmitem{ii} For the $\R$-ring $\opn{F}_{\mrm{bc}}(X, \R)$
of bounded continuous functions on a topological space $X$, 
the canonical norm is the sup norm. 

\rmitem{iii} The canonical norm makes $A$ into a Banach$^{\Ast}$ $\R$-ring. 
\end{itemize}
\end{thm}

Note that $\opn{F}_{\mrm{bc}}(X, \R)$ is indeed a BC $\R$-ring, by 
Theorem \ref{thm:415}.

\begin{proof}
Given an arbitrary BC $\R$-ring $A$, let $X := \opn{MSpec}(A)$
and $B := \opn{F}_{\mrm{bc}}(X, \R)$.
By Lemma \ref{lem:272}(3) the $\R$-ring homomorphism
$\opn{dev}_A : A \to \opn{F}_{\mrm{bc}}(X, \R) = B$
is bijective. The ring $B$ has the norm imposed by condition (ii),
and it is a Banach$^{\Ast}$ $\R$-ring; cf.\ Example \ref{exa:485}.
. 
The canonical norm of $A$ is defined to be the unique 
norm such that $\opn{dev}_A : A \to B$ is an 
isomorphism of Banach$^{\Ast}$ $\R$-rings. Conditions (i) and (ii) force this 
norm of $A$ to be unique. 

We need to prove functoriality, i.e.\ that condition (i) holds for an arbitrary 
homomorphism  $\phi : A \to B$ in $\cat{Rng} \xover{bc} \R$. 
Define the topological spaces $X := \opn{MSpec}(A)$ and 
$Y := \opn{MSpec}(B)$. In view of Lemma \ref{lem:251} and Proposition 
\ref{prop:450} there is an induced map 
$f := \opn{MSpec}(\phi) : Y \to X$ in $\cat{Top}$.
Consider the following diagram 
\begin{equation} \label{eqn:520}
\begin{tikzcd} [column sep = 12ex, row sep = 6ex]
A
\ar[r, "{\phi}"]
\ar[d, "{\opn{dev}_A}"', "{\simeq}"]
\ar[dd, bend right = 35, start anchor = west, end anchor = north west,
"{\opn{dev}_A}"']
&
B
\ar[d, "{\opn{dev}_B}", "{\simeq}"']
\ar[dd, bend left = 35, start anchor = east, end anchor = north east,
"{\opn{dev}_B}"]
\\
\opn{F}_{\mrm{bc}}(X, \R)
\ar[r, "{\opn{F}_{\mrm{bc}}(f, \R)}"]
\ar[d, "{\ep_X}"]
&
\opn{F}_{\mrm{bc}}(Y, \R)
\ar[d, "{\ep_Y}"']
\\
\opn{F}(X, \R)
\ar[r, "{\opn{F}(f, \R)}"]
&
\opn{F}(Y, \R)
\end{tikzcd} 
\end{equation}
in $\cat{Rng} \over \R$.  
The vertical arrows marked $\ep_X$ and $\ep_Y$ are the inclusions, and 
the bottom square is commutative, since $\opn{F}_{\mrm{bc}}(-, \R)$ is a 
subfunctor of  $\opn{F}(-, \R)$. 
The two outer paths from $A$ to $\opn{F}(Y, \R)$ are equal, by Proposition 
\ref{prop:453}. Because $\ep_Y$ is an injection, the top square is commutative 
too: 
$\opn{F}_{\mrm{bc}}(f, \R) \circ \opn{dev}_A = 
\opn{dev}_B \circ \, \phi$. 
The $\R$-ring homomorphism $\opn{F}_{\mrm{bc}}(f, \R)$ respects the sup 
norms on the function rings, and by construction the $\R$-ring 
isomorphisms $\opn{dev}_A : A \iso \opn{F}_{\mrm{bc}}(X, \R)$
and $\opn{dev}_B : B \iso \opn{F}_{\mrm{bc}}(Y, \R)$
respect the canonical norms on these rings. It follows that 
$\phi : A \to B$ respects the canonical norms. 
\end{proof}

\begin{cor} \label{cor:485}
The forgetful functor 
\[ F : \cat{BaRng}^{\Ast} \over \R \to \cat{Rng} \xover{bc} \R \]
is an equivalence of categories. 
\end{cor}

\begin{proof}
Given a BC $\R$-ring $B$, the canonical norm of $B$ makes it into a 
Banach$^{\Ast}$ $\R$-ring, which we denote by $G(B)$. 
Condition (i) in Theorem \ref{thm:486} says that 
\[ G : \cat{Rng} \xover{bc} \R \to \cat{BaRng}^{\Ast} \over \R \]
is a functor. It is clear that $F \circ G = \opn{id}$. A short calculation 
using the conditions in Theorem \ref{thm:486} shows that 
$G \circ F = \opn{id}$ too.
\end{proof}

\section{Stone Spaces}
\label{stone}

Recall that a topological space $X$ is called a {\em Stone space} if it is 
compact and totally disconnected (i.e.\ the only connected subsets of $X$ are 
the singletons). 
It is known that the category of Stone spaces is dual to the category of {\em 
boolean rings}. This is called {\em Stone duality}, see 
\cite[Corollary II.4.4]{Jo}, and it is very similar to 
Theorem \ref{thm:203} above. In this section we are going to study Stone spaces 
via their rings of continuous $\R$-valued functions. 

A topological space $X$ called a {\em profinite 
space} if $X \cong \opn{lim}_{\leftarrow i} X_i$, where 
$\{ X_i \}_{i \in I}$ is an inverse system of finite discrete spaces, and the 
inverse limit is taken in $\cat{Top}$. 

First we need to quote the next classical theorem. For proofs see 
\cite[Theorem 3.4.7]{BJ} or 
\cite[Lemma tag = 
\texttt{\href{https://stacks.math.columbia.edu/tag/08ZY}{08ZY}}]{SP}.

\begin{thm} \label{thm:495}
The following conditions are equivalent for a topological space $X$. 
\begin{enumerate}
\rmitem{i} $X$ is a Stone space. 

\rmitem{ii} $X$ is a profinite space.  
\end{enumerate}
\end{thm}

Suppose $X$ is a topological space and $Y \sub X$ is a subset. The {\em 
indicator function} of $Y$ is the function $1_Y : X \to \R$ 
given by the rule 
$1_Y(x) := 1$ if $x \in Y$ and $1_Y(x) := 0$ if $x \notin Y$.

\begin{lem} \label{lem:490}
Let $X$ be a topological space and $e : X \to \R$ a function. 
The following conditions are equivalent.
\begin{itemize}
\rmitem{i} $e$ is an idempotent, i.e.\ $e^2 = e$, and it is continuous. 

\rmitem{ii} $e = 1_Y$ for some open-closed subset $Y \sub X$.  
\end{itemize} 
\end{lem}

\begin{proof} \mbox{}

\smallskip \noindent
(i) $\Rightarrow$ (ii): An idempotent $\R$-valued function must take the 
values $1$ and $0$ only. Letting
$Y := \{ x \in X \mid e(x) = 1 \}$ we see that $e = 1_Y$.
Since $e$ and $1_X - e$ are continuous, the set 
$Y = \opn{NZer}_X(e) = \opn{Zer}_X(1_X - e)$ is open-closed. 

\medskip \noindent
(ii) $\Rightarrow$ (i): The function $1_Y \in A$ is idempotent and continuous.
\end{proof}

\begin{dfn} \label{dfn:487}
Let $X$ be a topological space and 
$A := \opn{F}_{\mrm{bc}}(X, \R)$. 
\begin{enumerate}
\item A function $a \in A$ is called a {\em step function} if $a$ takes only 
finitely many values. 

\item The set of step functions in $A$ is denoted by $A_{\mrm{stp}}$. 
\end{enumerate}
\end{dfn}

It is easy to see that $A_{\mrm{stp}}$ is an $\R$-subring of $A$. 

\begin{lem} \label{lem:491}
Let $X$ be a topological space and $A := \opn{F}_{\mrm{bc}}(X, \R)$. 
The following conditions are equivalent for $a \in A$. 
\begin{itemize}
\rmitem{i} $a$ is a step function.

\rmitem{ii} $a = \sum_i \la_i \cd e_i$, a finite sum, with $e_i \in A$ 
idempotents and $\la_i \in \R$.  

\rmitem{iii} $a$ factors through a finite discrete topological space $Z$. 
\end{itemize} 
\end{lem}

\begin{proof} \mbox{}

\smallskip \noindent
(i) $\Rightarrow$ (iii): Let $Z := a(x) \sub \R$. Then $Z$, with the induced 
topology from $\R$, is a finite discrete space. 

\medskip \noindent
(iii) $\Rightarrow$ (ii): Suppose 
$a = b \circ f$ with $f : X \to Z$ continuous. 
Letting $U_z := f^{-1}(z) \sub X$, we obtain a 
finite covering $X = \bigcup_{z \in Z} U_z$ by pairwise disjoint open-closed 
subsets. The function $e_z := 1_{U_z}$ is an idempotent element of $A$ 
by Lemma \ref{lem:490}; and 
$a = \sum_{z \in Z} b(z) \cd e_z$. 

\medskip \noindent
(ii) $\Rightarrow$ (i): The image of $a$ is contained in the finite set 
$\{ \la_i \} \sub \R$. 
\end{proof}

\begin{dfn} \label{dfn:495}
Let $X$ be a topological space, and let $f : X \to Z$ be a continuous map 
to a finite discrete topological space $Z$. For every $z \in Z$ let 
$V_z := f^{-1}(z)$. Then $\bsym{V} := \{ V_z \}_{z \in Z}$
is an open covering of $Z$, and we call it {\em the covering induced by 
$f : X \to Z$}.
\end{dfn}

Observe that an open covering $\{ V_z \}_{z \in Z}$ of $X$ induced by a 
map to a finite discrete space $Z$ is the same as a finite covering 
$\{ V_i \}_{i \in I}$ of $X$
such that each $V_i$ is open and closed, and such that 
$V_i \cap V_j = \varnothing$ for $i \neq j$. Such coverings are considered 
in 
\cite[Lemma tag = 
\texttt{\href{https://stacks.math.columbia.edu/tag/08ZZ}{08ZZ}}]{SP}.

Let $X$ be a topological space, and let 
$\bsym{U} = \{ U_i \}_{i \in I}$ and $\bsym{V} = \{ V_j \}_{j \in J}$
be open coverings of $X$. We say that $\bsym{V}$ is a refinement of $\bsym{U}$ 
is there is a function $\rho : J \to I$ such that 
$V_{j} \sub U_{\rho(i)}$ for every $j \in J$. 

\begin{lem} \label{lem:495}
Let $X$ be a Stone topological space, and let 
$\bsym{U} = \{ U_i \}_{i \in I}$
be an open covering of $X$. Then there is an open covering $\bsym{V}$ of $X$ 
that refines $\bsym{U}$, and $\bsym{V}$ is induced by a continuous map
$X \to Z$ to a finite discrete topological space $Z$. 
\end{lem}

\begin{proof}
Our proof is a slight modification of the proof of 
\cite[Lemma tag = 
\texttt{\href{https://stacks.math.columbia.edu/tag/08ZZ}{08ZZ}}]{SP}.
According to Theorem \ref{thm:495} there is a homeomorphism 
$X \cong \opn{lim}_{\leftarrow k} X_k$, where 
$\{ X_k \}_{k \in K}$ is an inverse system of finite discrete spaces.
Let $f_k : X \to X_k$ be corresponding map.  
We will prove that for some $k_0 \in K$ the covering induced by 
$f_{k_0} : X \to X_{k_0}$ refines $\bsym{U}$. 

Take some point $x \in X$. Because the space $X$ is homeomorphic to a closed 
subspace of the product $\prod_{k \in K} X_k$, there is a finite subset 
$K_x \sub K$ and an index $i_x \in I$ such that 
$x \in \bigcap_{k \in K_x} f_{k}^{-1}(f_{k}(x)) \sub U_{i_x}$.
But $K$ is a directed set, so there is some $k_x \in K$ that dominates the 
finite set $K_x$. This means that 
$x \in f_{k_x}^{-1}(f_{k_x}(x)) \sub U_{i_x}$.

We see that the collection 
$\bigl\{ f_{k_x}^{-1}(f_{k_x}(x)) \bigr\}_{x \in X}$ is an open covering of $X$ 
that refines $\bsym{U}$. Because $X$ is compact, we can pass to a subcovering  
$\bigl\{ f_{k_x}^{-1}(f_{k_x}(x)) \bigr\}_{x \in X_0}$
indexed by some finite subset $X_0 \sub X$. 

The finite set $\{ k_x \}_{x \in X_0}$ is dominated by some $k_0 \in K$.
For every $x \in X_0$ there is equality 
$f_{k_x}^{-1}(f_{k_x}(x)) = 
\bigcup_{y \in f_{k_0 / k_x}^{-1}(x)} f_{k_0}^{-1}(y) \sub X$,
where $f_{k_0 / k_x} : X_{k_0} \to X_{k_x}$ is the corresponding map. 
This means that 
$f_{k_0}(X) \sub  \bigcup_{x \in X_0} f_{k_0 / k_x}^{-1}(f_{k_x}(x))$,
and that 
$\bigl\{ f_{k_0}^{-1}(y) \bigr\}_{y \in f_{k_0}(X)}$ 
is a covering of $X$ refines $\bsym{U}$. 
But this covering is precisely the covering induced 
by $f_{k_0} : X \to X_{k_0}$. 
\end{proof}

\begin{thm} \label{thm:490}
Let $A$ be a BC $\R$-ring and $X := \opn{MSpec}(A)$. 
The following two conditions are equivalent:
\begin{itemize}
\rmitem{i} $X$ is a Stone space. 

\rmitem{ii} The subring $A_{\mrm{stp}}$ is dense in $A$, with respect to the 
canonical norm of $A$. 
\end{itemize}
\end{thm}

\begin{proof}
The topological space $X$ is compact by Theorem 
\ref{thm:203} (or, more precisely, by Lemma \ref{lem:272}(1)).
For convenience we identify the ring $A$ with the ring 
$\opn{F}_{\mrm{bc}}(X, \R)$, via the $\R$-ring isomorphism $\opn{dev}_A$, see 
Lemma \ref{lem:272}(3). Thus every element $a \in A$ is seen as a 
continuous function $a : X \to  \R$. The canonical norm on the 
Banach ring $A$ coincides with the sup norm on the compact space $X$. 

\medskip \noindent
(ii) $\Rightarrow$ (i): We already noted that $X$ is compact. It remains to 
prove that $X$ is totally disconnected. 
In other words, given a connected subset $Y$ of $X$, we must prove that $Y$ is 
a singleton. 

If $a \in A$ is a step function then $a(X)$ is a finite discrete subspace of 
$\R$. Because $Y$ is connected, the set $a(Y) \sub \R$ has to be a singleton.
We conclude that the function $a|_Y$ is constant.

Next take an arbitrary function $a \in A$. Given $\ep > 0$, condition (ii) says 
that there is a step function $a'$ such that $\norm{a - a'} < \ep$. We already 
know that $a'|_Y$ is constant. It follows that $a|_Y$ is constant. 

By Theorem \ref{thm:245} the elements of $A$ separate the points of $X$. 
We have just shown that every $a \in A$ is constant on $Y$. 
It follows that $Y$ has only one element.

\medskip \noindent
(i) $\Rightarrow$ (ii): Take an arbitrary $a \in A$ and $\ep > 0$. We need to 
produce a step function $a' \in A$ such that 
$\norm{a - a'} < \ep$. 

The function $a$ is bounded. Let $[\la_0, \la_1]$ be a finite closed 
interval in $\R$ that contains $a(X)$. Write $\la := \la_1 - \la_0$. 
Choose $k \in \N$ large enough such that
$2^{-k} \cd \la < 2^{-1} \cd \ep$. Define the numbers 
$\mu_i := \la_0 + i \cd 2^{-k} \cd \la \in \R$, so 
$\mu_0 = \la_0$ and $\mu_{2^k} = \la_1$. 
Next define the indexing set $I := \{ 0, \ldots, 2^k \}$. 
For every $i \in I$ define the open interval 
$W_i := (\mu_{i - 1}, \mu_{i + 1})$ in $\R$. 
Then $[\la_0, \la_1] \sub \bigcup_i W_i$, and the length of $W_i$ is $< \ep$. 

We now move to the space $X$. For every $i \in I$
let $U_i := a^{-1}(W_i) \sub X$. The collection of open sets 
$\bsym{U} := \{ U_i \}_{i \in I}$ is an open covering of $X$. 
According to Lemma \ref{lem:495} there is a continuous map 
$f : X \to Z$ to a finite discrete space $Z$, such that associated 
open covering $\bsym{V} := \{ V_z \}_{z \in Z}$,
where $V_z := f^{-1}(z)$, refines $\bsym{U}$. 
So there is a function $\rho : Z \to I$ such that 
$V_z \sub U_{\rho(z)}$ for all $z \in Z$. 
Let $b : Z \to \R$ be the function $b(z) := \mu_{\rho(z)}$. 
Finally let $a' := b \circ f : X \to \R$. Clearly $a'$ is a step function. 
For every point $x \in V_z$ we have $a(x), a'(x) \in U_{\rho(z)}$,
so $\abs{a(x) - a'(x)} < \ep$. But
$X = \bigcup_{z \in Z} V_z$, so $\norm{a - a'} < \ep$, as required. 
\end{proof}

The next corollary is classical. However, all the proofs we found in the 
literature are indirect and quite difficult. We give a straightforward proof, 
using Theorem \ref{thm:490}. 

\begin{cor} \label{cor:515}
Suppose $X$ is a discrete topological space, with Stone-\v{C}ech 
Compactification $\bar{X}$. Then $\bar{X}$ is a Stone topological space. 
\end{cor}

\begin{proof}
Let's write $A :=  \opn{F}_{\mrm{bc}}(X, \R)$.
According to Theorem \ref{thm:285}, we may assume that 
$\bar{X} = \opn{MSpec}(A)$, and the compactification map 
$\opn{c}_X : X \to \bar{X}$ is $\opn{refl}^{\mrm{\lmsp alg}}_X$. 
Let us write $\bar{A} := \opn{F}_{\mrm{bc}}(\bar{X}, \R)$. 
By Theorem \ref{thm:490} it is enough to prove that $\bar{A}_{\mrm{stp}}$ is 
dense in $\bar{A}$ for its canonical norm, which is the sup norm. 

The $\R$-ring homomorphism $\opn{dev}_A : A \iso \bar{A}$ 
is bijective, see Lemma \ref{lem:272}(3).
Since $\opn{c}_X(X)$ is dense inside $\bar{X}$, the sup norm of 
a function $a \in A$, when calculated on $X$, equals the sup norm of 
$\bar{a} = \opn{dev}_A(a) \in \bar{A}$, when calculated on $\bar{X}$. 
We conclude that it suffices to prove that $A_{\mrm{stp}}$ is dense in $A$ for 
the sup norm on $X$.

Take some function $a \in A$ and some $\ep > 0$. We need to produce a step 
function $a'$ such that $\norm{a - a'} < \ep$ in the sup norm on $X$. 
Now the set $a(X)$ is contained 
in some bounded closed interval $Z = [\la_0, \la_1] \sub \R$
of length $\la := \la_1 - \la_0$. Take $k \in \N$ sufficiently large such that 
$2^{-k} \cd \la < \ep$. Define the numbers 
$\mu_i := \la_0 + i \cd 2^{-k} \cd \la$, and the intervals 
$Z_1 := [\mu_{0}, \mu_{1}]$ and 
$Z_i := (\mu_{i - 1}, \mu_{i}]$ for $2 \leq i \leq 2^k$, so 
$\{ Z_i \}_{1 \leq i \leq 2^k}$ is a partition of $Z$ into 
intervals of length $< \ep$. 
Let $a' : X \to \R$ be the step function that has the value $\mu_i$ 
on the subset $a^{-1}(Z_i) \sub X$. 
Then $\norm{a - a'} < \ep$ as required.
\end{proof}

\begin{rem} \label{rem:495}
Stone spaces have become more interesting recently, since they form the 
background upon which {\em condensed mathematics} is built; see the notes 
\cite{Sc} by Scholze. In this theory the Stone-\v{C}ech compactifications of 
discrete spaces play an important role. 
\end{rem}

\section{Rings of Bounded Continuous Complex Valued Functions}
\label{BC-C-rings}

Here we introduce {\em BC $\C$-rings}, which are the complex variant of BC 
$\R$-rings. The main result is Theorem \ref{thm:435}; it states that BC 
$\C$-rings admit canonical involutions. Theorem \ref{thm:475} then says that 
taking canonical hermitian subrings is an equivalence from the category of BC 
$\C$-rings to the category of BC $\R$-rings. 

In analogy to Definition \ref{dfn:200}, given a topological space 
$X$, the following $\C$-rings exist: the ring $\opn{F}(X, \C)$ of all functions 
$a : X \to \C$, the ring $\opn{F}_{\mrm{c}}(X, \C)$ of continuous functions 
$a : X \to \C$, and the ring $\opn{F}_{\mrm{bc}}(X, \C)$ of bounded continuous 
functions $a : X \to \C$. Here continuity and boundedness are with respect to 
the standard norm on the field $\C$. As in the real case, the rings 
$\opn{F}_{}(X, \C)$, $\opn{F}_{\mrm{c}}(X, \C)$ and $\opn{F}_{\mrm{bc}}(X, \C)$ 
are not topologized. If $X$ is a discrete topological space then 
$\opn{F}_{\mrm{c}}(X, \C) = \opn{F}_{}(X, \C)$, 
and if $X$ is a compact topological space then 
$\opn{F}_{\mrm{bc}}(X, \C) = \opn{F}_{\mrm{c}}(X, \C)$. 

\begin{dfn} \label{dfn:420}
A $\C$-ring $A$ is called a {\em BC $\C$-ring} if it is isomorphic,
as a $\C$-ring, to the ring $\opn{F}_{\mrm{bc}}(X, \C)$ for some compact 
topological space $X$. 
The full subcategory of $\cat{Rng} \over \C$ on the BC $\C$-rings is denoted by 
$\cat{Rng}\xover{bc} \C$.
\end{dfn}

Thus, like in the real case, the  category $\cat{Rng}\xover{bc} \C$ is 
the essential image of the functor 
$\opn{F}_{\mrm{bc}}(-, \C) : (\cat{Top}_{\mrm{cp}})^{\mrm{op}} \to 
\cat{Rng} \over \C$.

The next definition is the complex analogue of Definitions \ref{dfn:230} and 
\ref{dfn:231}.

\begin{dfn} \label{dfn:423}
Let $X$ be a topological space.
\begin{enumerate}
\item For a point $x \in X$, let
$\opn{ev}_{x} : \opn{F}_{\mrm{bc}}(X, \C) \to \C$
be the $\C$-ring homomorphism $\opn{ev}_{x}(a) := a(x)$. It is called the 
{\em evaluation homomorphism}. 

\item Define the map of sets
$\opn{refl}^{\mrm{\lmsp alg}}_X : X \to 
\opn{MSpec} \bigl( \opn{F}_{\mrm{bc}}(X, \C) \bigr)$
to be $x \mapsto \opn{Ker}(\opn{ev}_{x})$. This map is called the {\em algebraic
reflection map}. 
\end{enumerate}
\end{dfn}

Here is the complex analogue of Lemma \ref{lem:232}.

\begin{lem} \label{lem:420} 
Let $X$ be a compact topological space, and let 
$A := \opn{F}_{\mrm{bc}}(X, \C)$.
\begin{enumerate}
\item Given $\m \in \opn{MSpec}(A)$, there is a point $x \in X$ such 
that $\m = \opn{Ker}(\opn{ev}_{x}) = \opn{refl}^{\mrm{\lmsp alg}}_X(x)$. 

\item The algebraic reflection map 
$\opn{refl}^{\mrm{\lmsp alg}}_X : X \to \opn{MSpec}(A)$
is a homeomorphism.
\end{enumerate}
\end{lem}

\begin{proof} \mbox{}

\smallskip \noindent
(1) Unlike the real case, which is well-known, we could not find a reference 
for the complex case, so here is the proof. 

Assume for the sake of 
contradiction that no such point $x$ exists. This means that for every 
$x \in X$ there is some element $a \in \m$ such that $a(x) \neq 0$. 
By continuity of $a$, there is an open neighborhood $U$ of $x$ such that 
$a(x') \neq 0$ for all $x' \in U$. Therefore we can find an open covering 
$X = \bigcup_{i \in I} U_i$, and elements $a_i \in \m$, such that 
$U_i \sub \opn{NZer}_X(a_i)$ for all $i$. The compactness of $X$ implies that 
there is a subcovering $X = \bigcup_{i \in I_0} U_i$ indexed by some finite 
subset $I_0 \sub I$. For every $i \in I_0$ the conjugate function 
$a_i^{\Ast}$, namely $a_i^{\Ast}(x) := \ol{a_i(x)}$, belongs to 
$A$. Therefore the function 
$b := \sum_{i \in I_0} a_i^{\Ast} \cd a_i$ belongs to the maximal ideal $\m$. 
The function $a_i^{\Ast} \cd a_i$ is positive on $U_i$, and hence the function 
$b$ is positive on all of $X$, and thus nonzero. The multiplicative inverse 
$b^{-1} : X \to \C$ is continuous, and by compactness it is bounded. Therefore 
the function $b$ is invertible in the ring $A$. This contradicts the 
fact that $b \in \m$. 

\medskip \noindent
(2) Let's write $\bar{X} := \opn{MSpec}(A)$.
Because $\opn{F}_{\mrm{bc}}(X, \R) \sub A$, Lemma \ref{lem:230}(1) implies 
that the map $\opn{refl}^{\mrm{\lmsp alg}}_X : X \to \bar{X}$ is injective.  
Item (1) above shows that the map $\opn{refl}^{\mrm{\lmsp alg}}_X$ is 
surjective. We see that $\opn{refl}^{\mrm{\lmsp alg}}_X : X \to \bar{X}$ is 
bijective.

It remains to prove that $\opn{refl}^{\mrm{\lmsp alg}}_X : X \to \bar{X}$ is a 
homeomorphism. Given an element $a \in A$, consider the principal open sets 
$U := \opn{NZer}_X(a) \sub X$ and 
$\bar{U} := \opn{NZer}_{\bar{X}}(a) \sub \bar{X}$. An easy calculation shows 
that $\opn{refl}^{\mrm{\lmsp alg}}_X(U) = \bar{U}$. We know (see Section 
\ref{sec:prlim}) that the principal open sets $\bar{U}$ form a basis of the 
Zariski topology of $\bar{X}$.
On the other hand, since $\opn{F}_{\mrm{bc}}(X, \R) \sub A$, Lemma 
\ref{lem:230}(2) implies that principal open sets form a basis of the given 
topology of $X$. Therefore $\opn{refl}^{\mrm{\lmsp alg}}_X$ is a homeomorphism.
\end{proof}

\begin{lem} \label{lem:421}
If $A$ is a BC $\C$-ring then it is $\C$-valued.
\end{lem}

\begin{proof}
We can assume that $A = \opn{F}_{\mrm{bc}}(X, \C)$ for some compact topological 
space $X$. Let $\m$ be a maximal ideal of $A$. 
By Lemma \ref{lem:420}(1) there is a point $x \in X$ such that 
$\m =  \opn{Ker}(\opn{ev}_x)$. But $\opn{ev}_x$ is a $\C$-ring homomorphism 
$\opn{ev}_x : A \to \C$, so it induces a $\C$-ring isomorphism 
$A / \m \iso \C$. 
\end{proof}

Next is the complex analogue of Theorem \ref{thm:415}.

\begin{thm} \label{thm:456}   
For an arbitrary topological space $X$, the ring
$\opn{F}_{\mrm{bc}}(X, \C)$ is a BC $\C$-ring.
\end{thm}

\begin{proof}
We may assume $X$ is a nonempty topological space. 
Consider the Stone-\v{C}ech Compactification
$(\bar{X}, \opn{c}_X)$ of $X$ (see Definition \ref{dfn:131} and Theorem 
\ref{thm:285}), and the $\C$-ring homomorphism 
\[ \opn{F}_{\mrm{bc}}(X, \opn{c}_X) : \opn{F}_{\mrm{bc}}(\bar{X}, \C) \to 
\opn{F}_{\mrm{bc}}(X, \C) . \]
Given a function $a \in \opn{F}_{\mrm{bc}}(X, \C)$, let
$Z \sub \C$ be a closed disc in $\C$ that contains $a(X)$. Because $Z$ is a 
compact topological space, the universal property of the SCC says that $a$ 
extends uniquely to a continuous function 
$\bar{a} : \bar{X} \to Z$. The assignment $a \mapsto \bar{a}$ is an inverse of 
the ring homomorphism $\opn{F}_{\mrm{bc}}(X, \opn{c}_X)$, so the latter is 
an isomorphism. But the topological space $\bar{X}$ is compact, 
and hence $\opn{F}_{\mrm{bc}}(\bar{X}, \C)$ is a BC $\C$-ring. 
\end{proof}

\begin{lem} \label{lem:422}
Let $A \in \cat{Rng}\xover{bc} \C$ and 
$X := \opn{MSpec}(A)$. 
\begin{enumerate}
\item The topological space $X$ is compact.  

\item There is a BC $\R$-ring $A_0$, with an isomorphism of $\C$-rings
$\C \ot_{\R} A_0 \cong A$. 

\item For every $a \in A$ the function $\opn{dev}_A(a) : X \to \C$ is 
continuous and bounded with respect to the standard norm of $\C$ .

\item The $\C$-ring homomorphism 
$\opn{dev}_A : A \to \opn{F}_{\mrm{bc}}(X, \C)$
is bijective. 
\end{enumerate}
\end{lem}

\begin{proof} \mbox{}

\smallskip \noindent
(1) This is like the proof of Lemma \ref{lem:272}(1), but now we use Lemma 
\ref{lem:420} (2) instead of Lemma \ref{lem:232}(2).

\medskip \noindent
(2) By definition there is a $\C$-ring isomorphism 
$\psi : A \iso B$ where $B := \opn{F}_{\mrm{bc}}(Y, \C)$ for some compact 
topological space $Y$. The ring
$B_0 := \opn{F}_{\mrm{bc}}(Y, \R)$ is a BC $\R$-ring, and the 
inclusion $\ga_B : B_0 \to B$ induces a $\C$-ring isomorphism 
$\C \ot_{\R} B_0 \iso B$.  
Then the subring $A_0 := \psi^{-1}(B_0) \sub A$ is also a BC $\R$-ring,
and the inclusion $\ga_A : A_0 \to A$ induces a $\C$-ring isomorphism 
$\C \ot_{\R} A_0 \iso A$.

\medskip \noindent
(3) We continue with $Y$, $B$, $B_0$ and $A_0$ as above. According to 
Theorem \ref{thm:203} there is a homeomorphism 
$\opn{refl}^{\mrm{\lmsp alg}}_Y : Y \iso \opn{MSpec}(B_0)$.
Define $X_0 := \opn{MSpec}(A_0)$. Theorem \ref{thm:203} says that 
the isomorphism $\psi_0 : A_0 \iso B_0$ in $\cat{Rng}\xover{bc} \R$
induces a homeomorphism $f := \opn{MSpec}(\psi_0) : Y \iso X_0$. By Theorem  
\ref{thm:458}(3) there is a homeomorphism 
$p =  \opn{MSpec}(\ga_A) : X \iso X_0$. By Proposition \ref{prop:453} 
the homomorphisms $\opn{dev}_{(-)}$ are functorial on 
$\cat{Rng}\xover{bc} \C$, 
so we may as well identify $X$ and $X_0$ via $p$. By 
Theorem \ref{thm:458}(4) there is this commutative diagram 
\begin{equation} \label{eqn:430}
\begin{tikzcd} [column sep = 10ex, row sep = 6ex]
A_0
\ar[r, "{\ga_{A}}"]
\ar[d, "{\opn{dev}_{A_0}}"', "{\simeq}"]
&
A
\ar[d, "{\opn{dev}_{A}}"]
\\
\opn{F}(X, \R)
\ar[r, "{\opn{F}(X, \ga_{\C})}"]
&
\opn{F}(X, \C)
\end{tikzcd}
\end{equation}
in $\cat{Rng} \over \R$. Here $\ga_{\C} : \R \to \C$ is the inclusion. 

Now the element $a \in A$ can be expressed as 
$a = a_0 + b_0 \cd \bi$ with $a_0, b_0 \in A_0$. 
The commutative diagram (\ref{eqn:430}) says that 
\[ \opn{dev}_{A}(a) = \opn{dev}_{A_0}(a_0) + 
\opn{dev}_{A_0}(b_0) \cd \bi \in \opn{F}(X, \C) . \]
By Lemma \ref{lem:272}(3) the functions 
$\opn{dev}_{A_0}(a_0), \opn{dev}_{A_0}(b_0) : X \to \R$
are continuous and bounded. It follows that $\opn{dev}_{A}(a)$ is continuous 
and bounded.

\medskip \noindent
(4) It remains to prove that 
$\opn{dev}_A : A \to \opn{F}_{\mrm{bc}}(X, \C)$
is bijective. By item (3) we know that there is the commutative diagram 
(\ref{eqn:430}) in $\cat{Rng} \over \R$, and the the left vertical arrow in it 
is a bijection.
We now pass to the induced commutative diagram in $\cat{Rng} \over \C$
\begin{equation} \label{eqn:434}
\begin{tikzcd} [column sep = 16ex, row sep = 6ex]
\C \ot_{\R} A_0
\ar[r, "{\opn{id}_{\C} \ot \, \ga_{A}}", "{\simeq}"']
\ar[d, "{\opn{id}_{\C} \ot \, \opn{dev}_{A_0}}"', "{\simeq}"]
&
A
\ar[d, "{\opn{dev}_{A}}"]
\\
\C \ot_{\R} \opn{F}_{\mrm{bc}}(X, \R)
\ar[r, "{\opn{id}_{\C} \ot \, \opn{F}_{\mrm{bc}}(X, \ga_{\C})}", 
"{\simeq}"']
&
\opn{F}_{\mrm{bc}}(X, \C)
\end{tikzcd}
\end{equation}
We see that the right vertical arrow here is bijective. 
\end{proof}

Next is the complex version of Corollary \ref{cor:456}.

\begin{thm} \label{thm:524} 
Let $A$ be a $\C$-ring, with $X := \opn{MSpec}(A)$. 
The following two conditions are equivalent:
\begin{itemize}
\rmitem{i} $A$ is a BC $\C$-ring. 

\rmitem{ii} $A$ is a $\C$-valued ring, for every $a \in A$ the function 
$\opn{dev}_A(a) : X \to \C$ is continuous, and the $\C$-ring homomorphism 
$\opn{dev}_A : A \to \opn{F}_{\mrm{c}}(X, \C)$
is bijective. 
\end{itemize}
\end{thm}

\begin{proof} \mbox{}

\smallskip \noindent
(i) $\Rightarrow$ (ii): 
By Lemma \ref{lem:422}(1) the topological space $X$ is compact, and therefore
$\opn{F}_{\mrm{bc}}(X, \C) = \opn{F}_{\mrm{c}}(X, \C)$.
According to Lemma \ref{lem:422}(3), for every $a \in A$ the function 
$\opn{dev}_A(a) : X \to \C$ is 
continuous. Finally, by Lemma \ref{lem:422}(4) the homomorphism 
$\opn{dev}_A : A \to \opn{F}_{\mrm{bc}}(X, \C) = 
\opn{F}_{\mrm{c}}(X, \C)$
is bijective.

\medskip \noindent 
(ii) $\Rightarrow$ (i): Proposition \ref{prop:395} says that the topological 
space $X$ is quasi-compact. This implies that 
$\opn{F}_{\mrm{bc}}(X, \C) = \opn{F}_{\mrm{c}}(X, \C)$,
and hence there is an isomorphism of $\C$-rings
$\opn{dev}_A : A \iso \opn{F}_{\mrm{bc}}(X, \C)$. 
According to Theorem \ref{thm:456} , $A$ is a BC $\C$-ring. 
\end{proof}

\begin{exa} \label{exa:466}
Take the $\C$-ring $A := \C[t]$, the polynomial ring in one variable over $\C$,
and let $X := \opn{MSpec}(A)$. The ring $A$ is a $\C$-valued ring. The function 
$\opn{dev}_A(t) : X \to \C$ is a bijection, but it is not continuous. Theorem 
\ref{thm:524} says that $A$ is not a BC $\C$-ring. 
\end{exa}

Here is the complex version of Theorem \ref{thm:203}. 

\begin{thm}[Duality] \label{thm:420}
The functor 
\[ \opn{F}_{\mrm{bc}}(-, \C) : (\cat{Top}_{\mrm{cp}})^{\mrm{op}} \to 
\cat{Rng}\xover{bc} \C \]
is an equivalence of categories, with quasi-inverse $\opn{MSpec}$.
\end{thm}
 
\begin{proof}
The proof is almost identical to the proof of Theorem \ref{thm:203}.
By Lemma \ref{lem:420}(2) the map  
$\opn{refl}^{\mrm{\lmsp alg}}_X : X \to \opn{MSpec}(A)$
is a homeomorphism, and the complex version of Lemma \ref{lem:258} shows that 
$\opn{refl}^{\mrm{\lmsp alg}}_X$ is functorial in $X$. 

In the reverse direction, Lemma \ref{lem:422}(4) shows that the $\C$-ring 
homomorphism 
$\opn{dev}_A : A \to \opn{F}_{\mrm{bc}}(X, \C)$
is bijective. By Proposition \ref{prop:453} the homomorphism 
$\opn{dev}_A$ is functorial in $A$. 
\end{proof}

\section{Involutive Complex Rings}
\label{inv-c}

In this section we prove that every BC $\C$-ring $A$ admits a canonical 
involution. This is Theorem \ref{thm:435}. The canonical involution is used to 
prove Theorem \ref{thm:475}, which states that the categories 
$\cat{Rng} \xover{bc} \C$ and $\cat{Rng} \xover{bc} \R$ are equivalent.

\begin{dfn}[Involutive $\C$-Rings] \label{dfn:435} \mbox{}
\begin{enumerate}
\item An {\em involution} of a $\C$-ring $A$ is an $\R$-ring automorphism 
$(-)^{\Ast} : A \to A$, satisfying $(a^{\Ast})^{\Ast} = a$ and
$(\la \cd a)^{\Ast} = \bar{\la} \cd a^{\Ast}$ 
for all  $a \in A$ and $\la \in \C$. Here $\bar{\la}$ is the complex 
conjugate of $\la$.

\item An {\em involutive $\C$-ring}
is a $\C$-ring $A$ equipped with an involution $(-)^{\Ast}$. 

\item Suppose $A$ and $B$ are involutive $\C$-rings. A {\em homomorphism of 
involutive $\C$-rings} $\phi : A \to B$ is a $\C$-ring homomorphism $\phi$ 
satisfying $\phi(a^{\Ast}) = \phi(a)^{\Ast}$ for every element $a \in A$.

\item The category of commutative involutive $\C$-rings, 
with involutive $\C$-ring homomorphisms, is denoted by 
$\cat{Rng}^{\Ast} \over \C$.  
\end{enumerate}
\end{dfn}

In functional analysis books, e.g.\ \cite{Co}, the names for the notions in 
items (2) and (3) above are {\em $\Ast$-algebra} and 
{\em $\Ast$-homomorphism}, respectively. 
(If $A$ is not commutative, then an involution has to be an anti-automorphism, 
i.e.\ $(a \cd b)^{\Ast} = b^{\Ast} \cd a^{\Ast}$.)

\begin{dfn} \label{dfn:436}
Suppose $A$ is an involutive commutative $\C$-ring.
Define the subset of {\em hermitian elements} of $A$ to be 
\[ A_0 := \{ a \in A \mid a^{\Ast} = a \} \sub A . \]
This is an $\R$-subring of $A$.  
\end{dfn}

\begin{exa} \label{exa:440}
Let $A_0$ be an $\R$-ring, and define the $\C$-ring
$A :=  \C \ot_{\R} A_0$. The ring $A$ has an involution $(-)^{\Ast}$, 
with formula $(\la \ot a_0)^{\Ast} := \bar{\la} \ot a_0$ for $\la \in \C$ and 
$a_0 \in A_0$. The hermitian subring of $A$ is $A_0$. 
\end{exa}

\begin{exa} \label{exa:441}
Let $X$ be a topological space, and define the $\C$-ring 
$A := \opn{F}_{\mrm{bc}}(X, \C)$. The ring $A$ has an involution $(-)^{\Ast}$, 
with formula $a^{\Ast}(x) := \ol{a(x)}$ for $a \in A$ and $x \in X$. The 
hermitian subring of $A$ is $A_0 := \opn{F}_{\mrm{bc}}(X, \R)$. 
\end{exa}

\begin{thm} \label{thm:435}
Every BC $\C$-ring $A$ admits a unique involution $(-)^{\Ast}$, called the 
{\em canonical involution}, satisfying the two conditions below. 
\begin{itemize}
\rmitem{i} Functoriality: Given a homomorphism $\phi : A \to B$ in 
$\cat{Rng}\xover{bc} \C$, and an element $a \in A$, there is equality
$\phi(a^{\Ast}) = \phi(a)^{\Ast}$ in $B$.

\rmitem{ii} Normalization: For the $\C$-ring $\opn{F}_{\mrm{bc}}(X, \C)$
of bounded continuous functions on a topological space $X$, 
the canonical involution is the one described in Example \ref{exa:441}.
\end{itemize}
\end{thm}

Note that $\opn{F}_{\mrm{bc}}(X, \C)$ is indeed a BC $\C$-ring, by 
Theorem \ref{thm:456}.

\begin{proof}
Given an arbitrary BC $\C$-ring $A$, let $X := \opn{MSpec}(A)$
and $B := \opn{F}_{\mrm{bc}}(X, \C)$.
By Lemma \ref{lem:422}(4) the $\C$-ring homomorphism
$\opn{dev}_A : A \to \opn{F}_{\mrm{bc}}(X, \C) = B$
is bijective. The ring $B$ has the involution imposed by condition (ii). 
The canonical involution of $A$ is defined to be the unique 
involution such that $\opn{dev}_A : A \to B$ is an 
isomorphism of involutive $\C$-rings. Conditions (i) and (ii) force this 
involution of $A$ to be unique.

We need to prove functoriality, i.e.\ that condition (i) holds for an arbitrary 
homomorphism  $\phi : A \to B$ in $\cat{Rng}\xover{bc} \C$. 
Define the topological spaces $X := \opn{MSpec}(A)$ and 
$Y := \opn{MSpec}(B)$. In view of Lemma \ref{lem:421} and Proposition 
\ref{prop:450} there is an induced map 
$f := \opn{MSpec}(\phi) : Y \to X$ in $\cat{Top}$.
Consider the following diagram 
\begin{equation*} 
\begin{tikzcd} [column sep = 12ex, row sep = 6ex]
A
\ar[r, "{\phi}"]
\ar[d, "{\opn{dev}_A}"', "{\simeq}"]
\ar[dd, bend right = 35, start anchor = west, end anchor = north west,
"{\opn{dev}_A}"']
&
B
\ar[d, "{\opn{dev}_B}", "{\simeq}"']
\ar[dd, bend left = 35, start anchor = east, end anchor = north east,
"{\opn{dev}_B}"]
\\
\opn{F}_{\mrm{bc}}(X, \C)
\ar[r, "{\opn{F}_{\mrm{bc}}(f, \C)}"]
\ar[d, "{\ep_X}"]
&
\opn{F}_{\mrm{bc}}(Y, \C)
\ar[d, "{\ep_Y}"']
\\
\opn{F}(X, \C)
\ar[r, "{\opn{F}(f, \C)}"]
&
\opn{F}(Y, \C)
\end{tikzcd} 
\end{equation*}
in $\cat{Rng} \over \C$ (which is similar to diagram (\ref{eqn:445})).  
The vertical arrows marked $\ep_X$ and $\ep_Y$ are the inclusions, and 
the bottom square is commutative, since $\opn{F}_{\mrm{bc}}(-, \C)$ is a 
subfunctor of  $\opn{F}(-, \C)$. 
The two outer paths from $A$ to $\opn{F}(Y, 
\C)$ are equal, by Proposition \ref{prop:453}.
Because $\ep_Y$ is an injection, the top square is commutative too: 
$\opn{F}_{\mrm{bc}}(f, \C) \circ \opn{dev}_A = 
\opn{dev}_B \circ \, \phi$. 
The $\C$-ring homomorphism $\opn{F}_{\mrm{bc}}(f, \C)$ respects the canonical 
involutions on the function rings, and by construction the $\C$-ring 
isomorphisms $\opn{dev}_A : A \iso \opn{F}_{\mrm{bc}}(X, \C)$
and $\opn{dev}_B : B \iso \opn{F}_{\mrm{bc}}(Y, \C)$
respect the canonical involutions on these rings. It follows that 
$\phi : A \to B$ respects the canonical involutions. 
\end{proof}

\begin{cor} \label{cor:455}
The procedure of equipping a BC $\C$-ring $A$ with its canonical involution 
$(-)^{\Ast}$ is a fully faithful functor 
\[ G : \cat{Rng}\xover{bc} \C  \to \cat{Rng}^{\Ast} \over \C . \]
\end{cor}

\begin{proof}
Clear from condition (i) in Theorem \ref{thm:435}.
\end{proof}

\begin{dfn} \label{dfn:437}
Given a BC $\C$-ring $A$, the hermitian subring of $A$, with respect to its 
canonical involution, is called the {\em canonical hermitian subring of 
$A$}. 
\end{dfn}

\begin{cor} \label{cor:445}
Let $A$ be a BC $\C$-ring
\begin{enumerate}
\item The canonical hermitian subring $A_0$ of $A$ is a BC $\R$-ring.

\item The ring homomorphism $\C \ot_{\R} A_0 \to A$, 
$\la \ot a_0 \mapsto \la \cd a_0$, is an isomorphism of involutive $\C$-rings. 
Here $A$ has the canonical involution, and $\C \ot_{\R} A_0$ has the  
involution from Example \ref{exa:440}.
\end{enumerate}
\end{cor}

\begin{proof} \mbox{}

\smallskip \noindent
(1) Let $X := \opn{MSpec}(A)$ and $B := \opn{F}_{\mrm{bc}}(X, \C)$.
Put on $A$ and $B$ their canonical involutions.
By Lemma \ref{lem:422} and Theorem \ref{thm:435} the homomorphism 
$\opn{dev}_A : A \to B$ is an isomorphism of $\C$-rings. 
Condition (i) of Theorem \ref{thm:435} says that $\phi$ is an isomorphism of 
involutive $\C$-rings. The hermitian subring of $B$ is 
$B_0 := \opn{F}_{\mrm{bc}}(X, \R)$, which is a BC $\R$-ring. 
Since $\opn{dev}_A$ induces an $\R$-ring isomorphism 
$A_0 \iso B_0$, we see that $A_0$ is a BC $\R$-ring. 

\medskip \noindent
(2) For $B = \opn{F}_{\mrm{bc}}(X, \C)$ as above the homomorphism 
\[ \C \ot_{\R} B_0 = \C \ot_{\R} \opn{F}_{\mrm{bc}}(X, \R) 
\xar{\lmsp \opn{id}_{\C} \ot \opn{F}_{\mrm{bc}}(X, \ga_{\C}) \lmsp}
\opn{F}_{\mrm{bc}}(X, \C) = B \]
is an isomorphism of involutive $\C$-rings. Because 
$\opn{dev}_A : A \to B$ is an isomorphism of involutive $\C$-rings, 
it follows that $\C \ot_{\R} A_0 \to A$ is also an  isomorphism of involutive 
$\C$-rings.
\end{proof}

\begin{thm} \label{thm:475}  
The procedure of sending a BC $\C$-ring $A$ to its canonical hermitian subring 
$H(A)$ is an equivalence of categories
\[ H : \cat{Rng}\xover{bc} \C \to \cat{Rng}\xover{bc} \R . \]
The quasi-inverse of $H$ is the functor 
\[ I : \cat{Rng}\xover{bc} \R \to \cat{Rng}\xover{bc} \C \, , \
I(A_0) := \C \ot_{\R} A_0 . \]
\end{thm}

\begin{proof} \mbox{}

\smallskip \noindent 
Step 1. Let us show that $H$ is a functor. 
Consider a homomorphism $\phi : A \to B$ in $\cat{Rng}\xover{bc} \C$.
By Corollary \ref{cor:455} there is a homomorphism 
$G(\phi) : G(A) \to G(B)$ in the category $\cat{Rng}^{\Ast} \over \C$; 
this just means that $\phi$ respects the canonical involutions on $A$ and $B$. 
Passing to hermitian subrings we obtain an $\R$-ring homomorphism 
$H(\phi) : H(A) \to H(B)$. 
By Corollary \ref{cor:445}(1) the $\R$-rings 
$A_0 := H(A)$ and $B_0 := H(B)$ are BC $\R$-rings. We conclude that 
$H : \cat{Rng}\xover{bc} \C \to \cat{Rng}\xover{bc} \R$
is a functor. 

\medskip \noindent 
Step 2. Obviously the procedure 
$I(A_0) := \C \ot_{\R} A_0$ and 
$I(\phi_0) :=  \opn{id}_{\C} \ot \, \phi_0$
is a functor
$I : \cat{Rng} \over \R \to \cat{Rng} \over \C$. 
If $A_0$ is a BC $\R$-ring, then, taking $X := \opn{MSpec}(A_0)$,
we have a $\C$-ring isomorphism 
\begin{equation} \label{eqn:476}
 \C \ot_{\R} A_0 \cong \C \ot_{\R} \opn{F}_{\mrm{bc}}(X, \R) 
\xar{\lmsp \opn{id}_{\C} \ot \opn{F}_{\mrm{bc}}(X, \ga_{\C}) \lmsp}
\opn{F}_{\mrm{bc}}(X, \C) . 
\end{equation}
In view of Theorem \ref{thm:456} we see that $I(A_0)$ is a BC $\C$-ring. 
Therefore we obtain a functor
$I : \cat{Rng}\xover{bc} \R \to \cat{Rng}\xover{bc} \C$. 

\medskip \noindent 
Step 3. Take some $A_0 \in  \cat{Rng}\xover{bc} \R$.
The $\C$-ring isomorphism (\ref{eqn:476}), with conditions (i) and (ii) of 
Theorem \ref{thm:435}, say that the canonical involution on 
$I(A_0) = \C \ot_{\R} A_0$ coincides with that of Example \ref{exa:440}.
Therefore the homomorphism 
$A_0 \to \C \ot_{\R} A_0$, $a_0 \mapsto 1_{\C} \ot a_0$, 
induces an $\R$-ring isomorphism 
\begin{equation} \label{eqn:477}
\ze_{A_0} : A_0 \iso (H \circ I)(A_0) . 
\end{equation}
This isomorphism is functorial in $A_0$. 

\medskip \noindent 
Step 4. Now take some $A \in \cat{Rng}\xover{bc} \C$.
Corollary \ref{cor:445}(2) gives a $\C$-ring isomorphism 
\begin{equation} \label{eqn:478}
\eta_A : (I \circ H)(A) \iso A .  
\end{equation}
It is easy to verify that $\eta_A$ is functorial in $A$.
\end{proof}

\section{Involutive Complex Banach Rings }
\label{sec:inv-c-b}

In this final section of the paper we relate the results from the previous  
sections of the paper to the theory of {\em complex $\mrm{C}^{\Ast}$-algebras} 
from functional analysis. Our source for material on functional analysis is 
\cite{Co}. 

As before, all rings in this section are commutative unital, and all ring 
homomorphisms preserve units. 

The notions of a Banach $\R$-ring, and a Banach $\R$-ring 
homomorphism, were recalled in Definitions \ref{dfn:485} and \ref{dfn:505} . The 
complex variants are the same, just with $\C$ instead of $\R$. 
The conventional name is "commutative unital Banach algebra", see 
\cite[Definition VII.1.1]{Co}. 

\begin{dfn} \label{dfn:376} 
The category of commutative Banach $\C$-rings, with Banach $\C$-ring 
homomorphisms, is denoted by $\cat{BaRng} \over \C$. 
\end{dfn}

The next definition is a rephrasing of \cite[Definition VIII.1.1]{Co}. 

\begin{dfn} \label{dfn:377} 
A {\em commutative Banach$^{\Ast}$ $\C$-ring}, commonly known as a 
{\em commutative unital $\mrm{C}^{\Ast}$-algebra over $\C$}, is a commutative 
Banach $\C$-ring $A$, with a norm $\norm{-}$ and 
an involution $(-)^{\Ast}$, satisfying the extra condition
$\norm{a^{\Ast} \cd a} = \norm{a}^2$ for all $a \in A$.

The category of commutative Banach$^{\Ast}$ $\C$-rings, with 
involutive Banach $\C$-ring homomorphisms, as in Definitions \ref{dfn:505}
and \ref{dfn:435}(3), is denoted by 
$\cat{BaRng}^{\Ast} \over \C$. 
\end{dfn}

There are two obvious forgetful functors: the functor
$\cat{BaRng}^{\Ast} \over \C \to \cat{BaRng} \over \C$ that forgets the 
involutions, and the functor
$\cat{BaRng}^{\Ast} \over \C \to \cat{Rng}^{\Ast} \over \C$
that forgets the norms.

Here is an easy to prove classical fact:

\begin{prop}[{\cite[Proposition VIII.1.11$\text{(d)}$]{Co}}] \label{prop:385}
Let $A$ and $B$ be involutive commutative Banach $\C$-rings, and let 
$\phi : A \to B$ be a homomorphism of involutive $\C$-rings. 
Then $\norm{\phi(a)} \leq \norm{a}$ for all $a \in A$. Thus $\phi$ is a 
Banach $\C$-ring homomorphism. 
\end{prop}

Proposition \ref{prop:385} directly implies:

\begin{cor} \label{cor:385}
The forgetful functor 
$\cat{BaRng}^{\Ast} \over \C \to \cat{Rng}^{\Ast} \over \C$
is fully faithful. 
\end{cor}

A much stronger result is our Corollary \ref{cor505} below. 

Here are some examples of involutive commutative Banach $\C$-rings.

\begin{exa} \label{exa:375}
The field $\C$ is a Banach${}^{\Ast}$ $\C$-ring,
with norm $\norm{\la} := \abs{\la}$ and involution 
$\la^{\Ast} := \bar{\la}$.

Let $n$ be a positive integer, and let $A$ be the $\C$-ring
$A := \C^n = \C \times \cdots \times \C$,
with componentwise ring operations. Put on $A$ the norm 
$\norm{(\la_1, \ldots, \la_n)} := \sup_{i} \lmsp \abs{\la_i}$ 
and the involution  
$(\la_1, \ldots, \la_n)^{\Ast} := (\bar{\la}_1, \ldots, \bar{\la}_n)$.
Then $A$ is a Banach${}^{\Ast}$ $\C$-ring. In fact, this is a special case of 
Example \ref{exa:376}, taking for $X$ the finite discrete topological space 
$X := \{ 1, \ldots, n \}$. 
\end{exa}

\begin{exa} \label{exa:376}
Let $X$ be a compact topological space, and let 
$A := \opn{F}_{\mrm{c}}(X, \C)$, the $\C$-ring of continuous functions 
$a : X \to \C$, where $\C$ is given its standard norm topology. 
Put on $A$ the involution $(-)^{\Ast}$ induced by complex conjugation, namely 
$a^{\Ast}(x) := \ol{a(x)}$ for all $a \in A$ and $x \in X$. 
Also put on $A$ the sup norm, namely 
$\norm{a} := \opn{sup}_{x \in X} \abs{a(x)}$. 
Then $A$ is a Banach${}^{\Ast}$ $\C$-ring. 
See \cite[Example VIII.1.4]{Co}.

If $f : Y \to X$ is a map in $\cat{Top}_{\mrm{cp}}$, then 
\[ \opn{F}_{\mrm{c}}(f, \C) : \opn{F}_{\mrm{c}}(X, \C) \to
\opn{F}_{\mrm{c}}(Y, \C) \]
is a Banach${}^{\Ast}$ $\C$-ring homomorphism. 
Thus we have a functor 
\begin{equation} \label{eqn:375}
\opn{F}_{\mrm{c}}(-, \C) : (\cat{Top}_{\mrm{cp}})^{\mrm{op}} 
\to \cat{BaRng}^{\Ast} \over \C .
\end{equation}
\end{exa}

The main classical theorem about Banach${}^{\Ast}$ $\C$-rings is the 
next one, quoted from \cite[Theorem VIII.2.1]{Co} and translated to our 
language. 

\begin{thm}[Gelfand-Naimark] \label{thm:375}
Let $A$ be a commutative Banach$^{\Ast}$ $\C$-ring.
\begin{enumerate}
\item The set $X$ of maximal ideals of $A$, suitably topologized, is 
compact. 

\item The Gelfand transform $A \to \opn{F}_{\mrm{c}}(X, \R)$
is an isomorphism of Banach$^{\Ast}$ $\C$-rings.
\end{enumerate}
\end{thm}

\begin{rem} \label{rem:410}
Let $A$ and $X$ be as in Theorem \ref{thm:375}. The suitable topology on 
$X$ is called the weak${}^{\Ast}$ topology. One can show that the 
weak${}^{\Ast}$ topology on $X$ coincides with its Zariski 
topology, and that the Gelfand transform coincides with the double 
evaluation homomorphism $\opn{dev}_{A}$. We won't need these facts in our paper.
\end{rem}

The category $\cat{Rng}\xover{bc} \C$ of commutative BC $\C$-rings was 
introduced in Definition \ref{dfn:420}.

\begin{thm}[Canonical Norm] \label{thm:505} 
Every BC $\C$-ring $A$ admits a unique norm $\norm{-}$, called the 
{\em canonical norm}, satisfying the three conditions below. 
\begin{itemize}
\rmitem{i} Given a homomorphism $\phi : A \to B$ in 
$\cat{Rng} \xover{bc} \C$, and an element $a \in A$, the inequality
$\norm{\phi(a)} \leq \norm{a}$ holds. 

\rmitem{ii} For the $\C$-ring $\opn{F}_{\mrm{bc}}(X, \C)$
of bounded continuous functions on a topological space $X$, 
the canonical norm is the sup norm. 

\rmitem{iii} The $\C$-ring $A$, equipped with the canonical norm $\norm{-}$
and the canonical involution $(-)^{\Ast}$ from Theorem \ref{thm:435}, 
is a Banach$^{\Ast}$ $\C$-ring. 
\end{itemize}
\end{thm}

\begin{proof}
The proof of Theorem \ref{thm:486} is valid also in the complex case, except 
for a few minor changes: we now use Example \ref{exa:376} instead of Example 
\ref{exa:485}, and Lemma \ref{lem:421} instead of Lemma \ref{lem:251}.
Also diagram (\ref{eqn:520}) is now in the category $\cat{Rng} \over \C$. 
\end{proof}

\begin{cor} \label{cor505}
The functor 
\[ F : \cat{BaRng}^{\Ast} \over \C \to \cat{Rng}\xover{bc} \C \]
that forgets the norms and the involutions is an 
equivalence of categories. 
\end{cor}

\begin{proof}
Take $A \in \cat{BaRng}^{\Ast} \over \C$. 
According to Theorem \ref{thm:375} there is a 
$\C$-ring isomorphism $F(A) \cong \opn{F}_{\mrm{c}}(X, \C)$, where 
$X$ is a compact topological space. 
By Definition \ref{dfn:420} the $\C$-ring $F(A)$ is a BC $\C$-ring.

The quasi-inverse of $F$ is the functor 
$G : \cat{Rng}\xover{bc} \C  \to \cat{BaRng}^{\Ast} \over \C$,
which puts on a BC $\C$-ring $A_0$ the canonical involution from Theorem 
\ref{thm:435} and the canonical norm from Theorem \ref{thm:505} .
\end{proof}

\begin{exa} \label{exa:473}
Consider the ring $A := \C[t]$, the polynomial ring in one variable over $\C$.
In Example \ref{exa:466} we saw that $A$ is not a BC $\C$-ring. 
Corollary \ref{cor505} implies that $A$ does not admit a structure of 
Banach$^{\Ast}$ $\C$-ring.
\end{exa}

Let $A$ and $B$ be commutative involutive Banach $\C$-rings, and let
$\phi : A \to B$ be a $\C$-ring homomorphism. 
The classical Proposition \ref{prop:385} says that $\phi$ is a homomorphism of 
involutive Banach $\C$-rings (Definition \ref{dfn:377}) iff $\phi$ respects 
the involutions; namely the norms are automatically respected. 
Our final result says that much more is true: the involutions are also 
automatically respected. 

\begin{cor} \label{cor:458}
Let $A$ and $B$ be commutative Banach$^{\Ast}$ $\C$-rings, and let 
$\phi : A \to B$ be a $\C$-ring homomorphism. Then $\phi$ is a homomorphism of 
Banach$^{\Ast}$ $\C$-rings, i.e.\ $\phi$ respects the involutions and the 
norms. 
\end{cor}

\begin{proof}
Corollary \ref{cor505} implies that the forgetful functor 
\[ F : \cat{BaRng}^{\Ast} \over \C \to \cat{Rng} \over \C \]
is fully faithful. In plain terms this says that for 
$A, B \in \cat{BaRng}^{\Ast} \over \C$
the function 
\[ F : \opn{Hom}_{\cat{BaRng}^{\Ast} \over \C}(A, B) \to 
\opn{Hom}_{\cat{Rng} \over \C}(A, B) \]
is bijective. Thus the homomorphism 
$\phi \in \opn{Hom}_{\cat{Rng} \over \C}(A, B)$
actually belongs to 
$\opn{Hom}_{\cat{BaRng}^{\Ast} \over \C}(A, B)$. 
\end{proof}


\end{document}